\documentclass[12pt,leqno]{article}
\usepackage{amssymb,amsmath,amsthm}
\usepackage[all]{xy}
\usepackage[alphabetic,lite]{amsrefs}
\numberwithin{equation}{section}

\usepackage[top=2cm,bottom=2cm,left=3cm,right=3cm,marginparsep=0.3cm,marginparwidth=3cm,includefoot]{geometry}

\usepackage[colorlinks=true, pdfstartview=FitV, linkcolor=blue, citecolor=blue, urlcolor=blue]{hyperref}

\usepackage{enumerate}
\usepackage{mathrsfs}

\usepackage{accents}

\newenvironment{bleu}
{\relax\color{blue}}
{\hspace*{.3ex}\relax}
\newcommand{\bev}{\begin{bleu}}
\newcommand{\ev}{\end{bleu}}

\newcommand{\beb}{\begin{bleu}}
\newcommand{\eb}{\end{bleu}}

\newenvironment{rouge}
{\relax\color{red}}
{\hspace*{.3ex}\relax}

\newcommand{\ber}{\begin{rouge}}
\newcommand{\er}{\end{rouge}}

\newcommand{\define}[1]{\emph{#1}}
\newcommand{\spa}{\vspace{0.5ex}\noindent}
\newcommand{\clos}[1]{\overline{#1}}

\newcommand{\nc}{\newcommand}
\nc{\on}{\operatorname}
\newtheorem{theorem}{Theorem}[section]
\newtheorem{proposition}[theorem]{Proposition}
\newtheorem{lemma}[theorem]{Lemma}
\newtheorem{corollary}[theorem]{Corollary}
\newtheorem{erratum}[theorem]{Erratum}

\theoremstyle{definition}

\newtheorem{definition}[theorem]{Definition}
\newtheorem{notation}[theorem]{Notation}
\newtheorem{example}[theorem]{Example}

\newtheorem{remark}[theorem]{Remark}

\nc{\Prop}{\begin{proposition}}
\nc{\enprop}{\end{proposition}}
\nc{\Lemma}{\begin{lemma}}
\nc{\enlemma}{\end{lemma}}
\nc{\Cor}{\begin{corollary}}
\nc{\encor}{\end{corollary}}
\nc{\Def}{\begin{definition}}
\nc{\enDef}{\end{definition}}
\nc{\Th}{\begin{theorem}}
\nc{\entheorem}{\end{theorem}}

\newcommand{\C}{{\mathbb{C}}}
\newcommand{\N}{{\mathbb{N}}}
\newcommand{\R}{{\mathbb{R}}}

\newcommand{\Z}{{\mathbb{Z}}}

\nc{\forl}{[\mspace{-.3mu}[\hbar]\mspace{-.3mu}]}
\nc{\Ls}{(\mspace{-.3mu}(\hbar)\mspace{-.3mu})}

\newcommand{\cor}{{\bf k}}

\nc{\kor}[1][M]{\C_{#1}((\hbar))}
\nc{\koro}[1][M]{\C_{#1}[[\hbar]]}

\nc{\Dh}[1][M]{\shd_{#1}^\hbar}
\nc{\Dhh}[1][M]{\shd_{#1}\forl}
\nc{\Dhhl}[1][M]{\shd_{#1}\Ls}
\nc{\Dhhh}[1][M]{\mathscr{D}_{#1}[\hbar,\opb{\hbar}]}

\nc{\DA}[1][X]{\mathscr{D}^{\mathscr{A}}_{#1}}

\nc{\tim}{q}

\def\BBS{{\mathbb S}}

\def\BBV{\mathbb{V}}

\def\phi{{\varphi}}
\def\epsilon{\varepsilon}

\def\sha{\mathscr{A}}

\def\shb{\mathscr{B}}

\def\shd{\mathscr{D}}
\def\she{\mathscr{E}}

\def\shi{\mathscr{I}}

\def\shl{\mathscr{L}}
\def\shm{\mathscr{M}}

\def\sho{\mathscr{O}}

\newcommand{\ol}{\overline}
\newcommand{\bl}{\bigl(}
\newcommand{\br}{\bigr)}

\newcommand{\lp}{{\rm(}}
\newcommand{\rp}{{\rm)}}

\nc{\RR}{\mathrm{R}}
\nc{\LL}{\mathrm{L}}

\newcommand{\into}{\hookrightarrow}

\newcommand{\A}[1][X]{\mathscr{A}_{{#1}}}

\nc{\Dhl}[1][X]{\shd_{#1}\Ls}
\nc{\Dho}[1][X]{\mathscr{D}_{#1}^\hbar(0)}
\nc{\Ohl}[1][X]{\sho_{#1}\Ls}
\nc{\Oh}[1][X]{\sho_{#1}^\hbar}
\nc{\OOh}[1][X]{\sho_{#1}[[\hbar]]}

\newcommand{\HWo}[1][X]{\widehat{\mathscr{W}}_{#1}(0)}

\newcommand{\gr}{\mathrm{gr}_\hbar}

\newcommand{\Lie}[1][]{\operatorname{\mathsf{L}}\def\temp{#1}
\ifx\temp\empty\else^{(#1)}\fi}

\newcommand{\stkHom}[1][]{\mathfrak{Hom}_{\raise1.5ex\hbox to.1em{}#1}}

\newcommand{\Int}{{\rm Int}}
\newcommand{\loc}{{\rm loc}}

\nc{\oA}[1][X]{\omega_{#1}^{\,\mathscr{A}}}

\nc{\oo}[1][X]{\omega_{#1}}

\renewcommand{\to}[1][]{\xrightarrow[]{#1}}

\newcommand{\isoto}[1][]{\xrightarrow[#1]%
{{\raisebox{-.6ex}[0ex][-.6ex]{$\mspace{1mu}\sim\mspace{2mu}$}}}}

\newcommand{\To}[1][]{\xrightarrow[]{\mspace{10mu}{#1}\mspace{10mu}}}

\newcommand{\Hom}[1][]{\mathrm{Hom}_{\raise1.5ex\hbox to.1em{}#1}}
\newcommand{\RHom}[1][]{\RR\mathrm{Hom}_{\raise1.5ex\hbox to.1em{}#1}}
\newcommand{\Ext}[2][]{\mathrm{Ext}_{\raise1.5ex\hbox to.1em{}#1}^{#2}}
\renewcommand{\hom}[1][]{{\mathscr{H}\mspace{-4mu}om}_{\raise1.5ex\hbox to.1em{}#1}}
\newcommand{\rhom}[1][]{{\RR\mathscr{H}\mspace{-3mu}om}_{\raise1.5ex\hbox to.1em{}#1}}
\newcommand{\ext}[2][]{{\mathscr{E}\mspace{-2mu}xt}_{%
\raise1.5ex\hbox to.1em{}#1}^{#2}}
\newcommand{\BHom}[1][]{\mathrm{Bhom}_{\raise1.5ex\hbox to.1em{}#1}}

\newcommand{\Tens}[1][]{\mathbin{\otimes_{\raise1.5ex\hbox to-.1em{}{#1}}}}
\newcommand{\LTens}[1][]{\mathbin{\otimes_{\raise1.5ex\hbox to-.1em{}#1}^{L}}}
\newcommand{\Tor}[2][]{\mathrm{Tor}^{\raise1.5ex\hbox to.1em{}#1}_{#2}}
\newcommand{\tens}[1][]{\mathbin{\otimes_{\raise1.5ex\hbox to-.1em{}{#1}}}}

\newcommand{\dtens}[1][]%
{{\overset{\mathrm{L}}{\underline{\otimes}}}_{#1}}

\nc{\aut}{{\sha\mspace{-1mu}\mit{ut}}\,}

\newcommand{\shend}{\operatorname{{\she\mspace{-2mu}\mathit{nd}}}}
\newcommand{\Endo}[1][]{\mathrm{End}_{\raise1.5ex\hbox to.1em{}#1}}
\newcommand{\sendo}[1][]{{\shend}_{\raise1.5ex\hbox to.1em{}#1}}

\newcommand{\Aut}[1][]{\mathrm{Aut}_{\raise1.5ex\hbox to.1em{}#1}}
\newcommand{\sect}{\Gamma}
\newcommand{\rsect}{\mathrm{R}\Gamma}

\newcommand{\Rb}{{\rm b}}

\newcommand{\grad}{{\rm grad}}

\newcommand{\SSi}{\on{SS}}

\newcommand{\roim}[1]{\RR{#1}_*}
\newcommand{\reim}[1]{\RR{#1}_!}
\newcommand{\opb}[1]{#1^{-1}}
\newcommand{\epb}[1]{#1^{!}}

\newcommand{\tw}[1]{\widetilde{#1}}

\DeclareMathOperator{\supp}{supp}
\DeclareMathOperator{\ori}{or}
\DeclareMathOperator{\chv}{char}

\DeclareMathOperator{\hchv}{hypchar}

\newcommand{\Pro}{\mathrm{Pro}}

\newcommand{\eqdot}{\mathbin{:=}}

\newcommand{\scdot}{\,\cdot\,}
\newcommand{\cl}{\colon}
\newcommand{\scbul}{{\,\raise.4ex\hbox{$\scriptscriptstyle\bullet$}\,}}

\newcommand{\ba}{\begin{array}}
\newcommand{\ea}{\end{array}}

\newcommand{\bnum}{\begin{enumerate}[{\rm(i)}]}
\newcommand{\enum}{\end{enumerate}}
\newcommand{\banum}{\begin{enumerate}[{\rm(a)}]}
\newcommand{\eanum}{\end{enumerate}}

\newcommand{\eq}{\begin{eqnarray}}
\newcommand{\eneq}{\end{eqnarray}}
\newcommand{\eqn}{\begin{eqnarray*}}
\newcommand{\eneqn}{\end{eqnarray*}}

\nc{\Der}{\on{D}}
\nc{\Derb}{\mathrm{D}^{\mathrm{b}}}
\nc{\hs}{\hspace*}
\nc{\Supp}{\on{Supp}}
\nc{\tr}{\on{tr}}

\newcommand{\HHWo}[1][X]{\mathcal{HH}(\HWo[])}

\newcommand{\RD}{{\rm D}}
\nc{\RDAA}[1][X]{\mathrm{D}^\prime_{\mathscr{A}_{#1}}}
\nc{\RDA}{\mathrm{D}^\prime_{\mathscr{A}}}
\nc{\RDAh}{\mathrm{D}^\prime_{\mathscr{A}^{\rm loc}}}
\nc{\RDO}{\mathrm{D}^\prime_{\mathscr{O}}}
\nc{\RDOl}{\mathrm{D}^\prime_{\mathscr{O}^{\hbar,\rm loc}}}
\nc{\RDDO}{\mathrm{D}_{\mathscr{O}}}
\nc{\RDDA}{\mathrm{D}_{\mathscr{A}}}
\nc{\RDD}{\mathrm{D}^\prime}

\nc{\conv}[1][]{\mathop{\circ}\limits_{#1}}
\newcommand{\aconv}[1][]{\mathop{\circ}\limits^{a}\limits_{#1}}

\nc{\sconv}[1][]{\mathop{\ast}\limits_{#1}}

\nc{\ssum}{\mathop{\mbox{\normalsize$\sum$}}}
\nc{\de}[1][X]{\delta_{#1}} 
\nc{\vs}{\vspace}
\nc{\dA}[1][X]{{{\delta}_*\mspace{1mu}\mathscr{A}_{{#1}}}}
\nc{\dO}[1][X]{{{\delta_{{#1}}}_*\mspace{1mu}\mathscr{O}_{{#1}}}}
\nc{\dGA}[1][X]{\gr(\mathscr{C}_{{#1}})}

\nc{\OS}[1][X]{\sho_{#1}}
\nc{\soplus}{\mathop{\text{\scriptsize\raisebox{.5ex}{$\displaystyle\bigoplus$}}}}
\nc{\Inv}{\on{Inv}}
\nc{\stkInv}{\mathfrak{Inv}}
\nc{\bwr}{{\mbox{\large$\wr$}}}

\nc{\be}{\begin{enumerate}}
\nc{\ee}{\end{enumerate}}
\nc{\stan}{\mathrm{stan}}
\nc{\Db}{\RD^\Rb}
\nc{\pt}{\mathrm{pt}}
\nc{\BBD}{\mathbb{D}}
\nc{\rC}{\mathrm{C}}
\nc{\scup}{\mathop{\text{\scriptsize\raisebox{.5ex}{$\displaystyle\bigcup$}}}}
\nc{\tX}{{\widetilde{X}}}
\nc{\AL}{\A[\Lambda/X]}
\nc{\ALa}{\A[\Lambda^a]}
\nc{\Gr}{\on{Gr}}
\nc{\CL}{\on{\mathrm{C}_\Lambda}}
\nc{\codim}{\on{codim}}
\nc{\Chl}{\on{char_{\Lambda}}}
\nc{\Chlo}{\on{char_{\Lambda_0}}}
\nc{\ChM}{\on{char_{M}}}
\nc{\DL}{\shd_\shl}
\nc{\DLl}{\shd_\shl^\loc}
\nc{\rmC}{\rm C}

\nc{\Zh}{\Z\forl}
\nc{\PZ}{\Pro(\Zh)}
\nc{\coco}{{cohomologically complete}}
\nc{\rpi}{{{\rm R}\pi}}
\nc{\shal}{{\sha^\loc}}

\DeclareMathOperator{\pinto}{\Int_{\rm pw}}
\DeclareMathOperator{\pclo}{\rm{cl_{pw}}}
\newcommand{\pint}[1]{\pinto\left(#1\right)}
\newcommand{\pcl}[1]{\pclo\left(#1\right)}
\newcommand{\ccirc}{{\circ\circ}}

\newcommand{\pasii}[1]{J^-_{\preceq_i}{(#1)}}

\newcommand{\gDel}{\Delta_\gamma}

\newcommand{\pas}[1]{J^-_\gamma{(#1)}}

\newcommand{\preceql}{\mathop{\preceq_\gamma}}

\newcommand{\futai}[1]{I^+_{\gamma}{(#1)}}
\newcommand{\pasai}[1]{I^-_{\gamma}{(#1)}}

\newcommand{\futo}[1]{J^+_\preceq{(#1)}}
\newcommand{\paso}[1]{J^-_\preceq{(#1)}}
\newcommand{\futoa}[1]{J^+_{\preceq^a}{(#1)}}

\newcommand{\futi}[1]{I^+_\gamma{(#1)}}

\newcommand{\futps}[1]{J^+_{\rm ps}{(#1)}}
\newcommand{\pasps}[1]{J^-_{\rm ps}{(#1)}}
\newcommand{\preceqps}{\mathop{\preceq}_{\rm ps}}
\newcommand{\gDelps}{\Delta_{\rm ps}}

\newcommand{\gDelccc}{\Delta_{\rm cc}}
\newcommand{\futccc}[1]{J^+_{\rm cc}{(#1)}}
\newcommand{\pasccc}[1]{J^-_{\rm cc}{(#1)}}
\newcommand{\preceqccc}{\mathop{\preceq}_{\rm cc}}

\newcommand{\gPo}{\Delta_{\preceq}}
\newcommand{\gPoa}{\Delta_{\preceq^a}}

\newcommand{\sq}{\sqrt{-1}}

\begin{document}

\title{Sheaves and D-modules on Lorentzian manifolds}
\author{Beno\^it Jubin and Pierre Schapira%
}
\maketitle

\begin{abstract}
We introduce a class of causal manifolds which contains the globally hyperbolic spacetimes and we prove global propagation theorems for sheaves on such manifolds.
As an application, we solve globally the Cauchy problem for hyperfunction solutions of hyperbolic systems.
\end{abstract}

\tableofcontents

\section*{Introduction}

A causal manifold $(M,\gamma)$ is a real smooth manifold $M$ endowed with an everywhere nonempty open convex cone $\gamma$ in its tangent bundle $TM$.
The main examples of such manifolds are provided by Lorentzian spacetimes, that is, time-oriented Lorentzian manifolds.
Lorentzian spacetimes and the properties of their causal preorders are important objects of study in the mathematics of general relativity.
A natural problem in this field is to solve {\em globally} the Cauchy problem for the wave operator or for related operators, with initial data on a Cauchy hypersurface.
This problem has been and still is the object of an intense activity in mathematical physics (see for example~\cites{Ge70,HE73,BF00,BGP07,MS08,BF11}).
It was initiated in the pioneering work of Jean Leray~\cite{Le53}.

In {\bf Section~1} of this paper, using the tools of the Whitney normal cone and related notions as in~\cite{KS90} (see the appendix), we introduce the notion of a $\gamma$-set in $M$.
The family of $\gamma$-sets is stable by union and intersection, which allows us to define causal preorders on a causal manifold and in particular the cc-preorder, the finest closed causal preorder.
We then define a {\em Cauchy time function} $\tim\cl M\to\R$ (see Definition~\ref{def:Gcausal}) as a submersive causal map which is proper on the future and the past (for the cc-preorder) of any point.
The cc-preorder on a causal manifold with a Cauchy time function is proper (causal diamonds are compact).
By theorems of Geroch and Bernal--S\'anchez (\cites{Ge70,BS05}), globally hyperbolic Lorentzian spacetimes may be endowed with Cauchy times functions.
However, the situation here is more general: the cone $\gamma$ need not have a smooth boundary nor be quadratic, and a causal manifold with a Cauchy time function may have causal loops.
A triple $(M,\gamma,\tim)$ composed of a causal manifold endowed with a Cauchy time function is called here a G-causal manifold, by reference to Geroch.

In {\bf Section~2}, we apply the microlocal theory of sheaves of \cite{KS90} to causal manifolds and obtain propagation results for sheaves on G-causal manifolds.
More precisely, denote by $T^*M$ the cotangent bundle to $M$, by $\gamma^{\circ a} = -\gamma^\circ$ the opposite polar cone of $\gamma$, and by $T^*_MM$ the zero-section of $T^*M$.
Choose a field $\cor$ and consider an object $F$ of the bounded derived category $\Derb(\cor_M)$ of sheaves of $\cor$-modules on $M$.
Denote as usual by $\SSi(F)$ its microsupport, a closed conic co-isotropic subset of $T^*M$.
We prove here the following results (see Theorem~\ref{th:GglobalCP1}).

(1)
If $\SSi(F)\cap\gamma^{\circ a}\subset T^*_MM$, then for any compact set $K$ the restriction morphism
\eq\label{eq:intro1}
&&\rsect(M;F)\To\rsect(M\setminus\pasccc{K};F)
\eneq
is an isomorphism.
In other words, any ``section'' of $F$ on $M\setminus\pasccc{K}$ extends uniquely to $M$.

(2)
If $\SSi(F)\cap(\gamma^\circ \cup \gamma^{\circ a})\subset T^*_MM$, then setting $L = \opb{q}(0)$, the restriction morphism
\eq\label{eq:intro2}
&&\rsect(M;F)\To\rsect(L;F\vert_L)
\eneq
is an isomorphism.
In other words, any ``section'' of $F$ defined in a neighborhood of a Cauchy hypersurface $L$ extends uniquely to the whole of $M$.
In fact, we prove a more general result when replacing the Cauchy time function $q \colon M \to \R$ with a submersive morphism $f \colon M \to N$, the manifold $N$ being contractible (see Corollary~\ref{cor:GP3}).

Both isomorphisms~\eqref{eq:intro1} and~\eqref{eq:intro2} are easy consequences of Theorem~\ref{th:GP1} below which gives a bound to the microsupport of direct images in a non-proper situation.

In {\bf Section~3}, we apply the preceding results to the case where $M$ is real analytic and $F$ is the complex of hyperfunction (or analytic) solutions of a hyperbolic system $\shm$.
In~\cite{Sc13} (based on~\cite{KS90}), it is explained how the microlocal theory of sheaves allows one to solve the Cauchy problem and to give domains of propagation for the hyperfunction solutions of a linear hyperbolic system.
We translate these results here in the case of causal manifolds with Cauchy time functions and show that the Cauchy problem may be solved {\em globally}.
Note that the notion of hyperbolicity used here relies only on the characteristic variety of the system.
In case of a linear equation $Pu=v$, it corresponds to ``weakly hyperbolic'' in the old terminology.
Indeed, the framework of hyperfunctions is much more flexible than that of distributions as far as one wants to solve the Cauchy problem or to study analytic propagation.
We end this section with several detailed examples:
(1) in the case of a product of $\R$ (the ``time'') with a compact manifold, we give general sufficient conditions on a differential operator for the Cauchy problem to be globally well-posed;
(2) we extend these results to the case of a product of $\R$ with a complete Riemannian manifold;
(3) we give a propagation result for operators on a product of $\C$ (the ``complex time'') with a compact manifold;
(4) in the general case of a globally hyperbolic Lorentzian spacetime, we solve globally the Cauchy problem for operators of wave type.

\begin{erratum}\label{err:1}
The statement~\cite{DS98}*{Prop.~4.4~(ii)} is not correct \lp see Example~\ref{exa:DS}\rp.
Therefore,~\cite{Sc13}*{Prop.~6.6} as well as its corollaries are not correct.
However, and that is what we shall show in this paper, most of the applications to causal manifolds are correct when assuming the spacetime globally hyperbolic.
\end{erratum}

\section{Causal manifolds}

\subsection{Notation and background}\label{subsect:not}The statements  of this subsection are all elementary and well-known. We recall them  to fix some notations.

Unless otherwise specified, a manifold means a real $C^\infty$-manifold and a morphism of manifolds $f\cl M\to N$ is a map of class $C^\infty$.

Let $M$ be a manifold.
For any subset $A\subset M$, we denote by $\ol A$ its closure, by $\Int(A)$ its interior and we set $\partial A = \ol{A} \setminus \Int(A)$.

\subsubsection*{Cones  in vector bundles}
Recall that in a real  finite-dimensional vector space $\BBV$ a cone is \define{proper}   if its convex hull does not contain any nonzero linear subspace.
The dual (see~\eqref{eq:dualcone} below) of a cone is a convex closed cone.
The dual of a cone with nonempty interior is proper.

Let $p\cl E\to M$ be a real (finite-dimensional) vector bundle over $M$.
As usual, one denotes by $a\cl E\to E$ the antipodal map, $(x;\xi)\mapsto(x;-\xi)$.
A subset $\gamma$ of $E$ is \define{conic} (or is a cone) if it is invariant by the action of $\R_{>0}$, that is, $\gamma_x\subset E_x$ is a cone for each $x\in M$.
Here, $\gamma_x$ is the restriction of $\gamma$ to the fibre $E_x$.
If $\gamma$ is closed and conic, then its projection by $p$ on $M$ coincides with its intersection with the zero-section of $E$ and we identify the zero-section of $E$ with $M$.

Let $\gamma$ be a cone in $E$.
We denote by $\gamma^\circ$ the polar (or dual) cone in $E^*$ and by $\gamma^{a}$ the image of $\gamma$ by the antipodal map $a$:
\eq
&&\gamma^\circ=\{(x;\xi)\in E^*;\langle\xi,v\rangle\geq0\mbox{ for all }v\in\gamma_x\},\label{eq:dualcone}\\
&&\gamma^{a}=a(\gamma).\label{eq:opcone}
\eneq
For a cone $\gamma$ in a vector bundle $E$, one has to distinguish between its closure $\ol{\gamma}$ and its pointwise closure $\pcl{\gamma}$.
Similarly, one has to distinguish between its interior $\Int(\gamma)$ and its pointwise interior $\pint{\gamma}$.
One has
\eqn
&&\Int(\gamma) \subset \pint{\gamma} \subset \gamma \subset \pcl{\gamma} \subset \ol{\gamma}
\eneqn
and all inclusions may be strict.

\begin{example}\label{exa:cones}
Consider two nonempty open convex cones $\gamma_0^1\subsetneq\gamma_0^2$ in a vector space $\BBV$ and two nonempty open subsets $U\subsetneq V\subset\BBV$ and define
\eqn
&&\gamma=V\times\gamma_0^1\cup U\times\gamma_0^2.
\eneqn
Then, for $x\in V\cap\partial U$, one has $\ol{\gamma_x} = \ol{\gamma_0^1}$ and $\ol{\gamma}_x = \ol{\gamma_0^2}$.
Therefore, $\gamma^\ccirc = \pcl{\gamma}$ is not closed.
\end{example}

\begin{lemma}\label{le:lambdaclosed}
Let $p\cl E\to M$ be a real finite-dimensional vector bundle over $M$ and let $\gamma$ be an open convex cone.
Then $\gamma^\circ$ is a closed convex cone, $\gamma = \Int(\pcl{\gamma}) = \Int({\gamma^\ccirc})$.
\end{lemma}

\begin{proof}
(i)
First, $\gamma^\circ$ being a polar cone, it is convex.
We shall prove that it is closed.
Let $(x;\xi) \in \ol{\gamma^\circ}$.
If $\gamma_x=\varnothing$, then $(x;\xi) \in \gamma^\circ$.
Now assume $\gamma_x\neq\varnothing$.
We choose a local chart $U$ in a neighborhood of $x$ so that $E\vert_U=U\times\BBV$.
There exists a sequence $(x_n, \xi_n)_n \to[n] (x,\xi)$ with $(x_n,\xi_n) \in \gamma^\circ$.
Let $v \in \gamma_x$.
Since $\gamma$ is open, there exists an open cone $\theta\subset\BBV$ such that $(x,v)\in U \times \theta\subset\gamma$.
Then for all $n$ large enough, $x_n \in U$ and $\langle\xi_n,v\rangle \geq 0$.
Therefore, $\langle\xi,v\rangle \geq 0$ and $(x;\xi) \in \gamma^\circ$.

\spa
(ii)
Since $\gamma \subset \pcl{\gamma}$ and $\gamma$ is open, $\gamma \subset \Int(\pcl{\gamma})$.
On the other hand, $\Int(\pcl{\gamma}) \subset \pint{\pcl{\gamma}} = \gamma$ (recall that in a vector space, an open convex set is equal to the interior of its closure).
Therefore, $\gamma = \Int(\pcl{\gamma})$.

\spa
(iii)
The set $U = \{ x \in M ; \gamma_x \neq \varnothing \}$ is open, and $\gamma^\ccirc|_U = \pcl{\gamma}|_U$ while $\gamma^\ccirc|_{M \setminus U} = M \setminus U$ and $\pcl{\gamma}|_{M\setminus U} = \varnothing$.
Therefore, $\Int({\gamma^\ccirc}) = \Int(\pcl{\gamma})$.
\end{proof}

\subsubsection*{Cotangent bundles}

Let $M$ be a manifold.
We denote by $\tau\cl TM\to M$ and by $\pi\cl T^*M\to M$ its tangent and cotangent bundle, respectively.

For a submanifold $N$ of $M$, we denote by $T_N M = (N \times_M TM)/TN$ the normal bundle of $N$ in $M$ and by $T^*_N M$ its dual, the conormal bundle of $N$ in $M$.
In particular, $T^*_MM$ is the zero-section.

For two manifolds $M$ and $N$ we denote by $q_1$ and $q_2$ the first and second projections defined on $M\times N$.
We denote by $\Delta_M$, or simply $\Delta$, the diagonal of $M\times M$.

Let $M_i$ ($i=1,2,3$) be manifolds.
For short, we write $M_{ij}\eqdot M_i\times M_j$ ($1\leq i,j\leq3$) and $M_{123}=M_1\times M_2\times M_3$.
We denote by $q_i$ the projection $M_{ij}\to M_i$ or the projection $M_{123}\to M_i$
and by $q_{ij}$ the projection $M_{123}\to M_{ij}$.
For $A_1\subset M_{12}$ and $A_2\subset M_{23}$, one sets
\eqn
&&A_1\conv[2]A_2=q_{13}(\opb{q_{12}}A_1\cap\opb{q_{23}}A_2).
\eneqn
Similarly, we denote by
$p_i$ the projection $T^*M_{ij}\to T^*M_i$ or the projection $T^*M_{123}\to T^*M_i$
and by $p_{ij}$ the projection $T^*M_{123}\to T^*M_{ij}$.
We also need to introduce the map $p_{ij^a}$, the composition of $p_{ij}$ and the antipodal map on $T^*M_j$ and similarly with $p_{i^aj}$.

Let $\Lambda_1\subset T^*M_{12}$ and $\Lambda_2\subset T^*M_{23}$.
We set
\eq\label{eq:convolution_of_sets}
&&\Lambda_1\aconv[2]\Lambda_2\eqdot
p_{13}(\opb{p_{12^a}}\Lambda_1\cap\opb{p_{23}}\Lambda_2)
=p_{13}(\opb{p_{12}}\Lambda_1\cap\opb{p_{2^a3}}\Lambda_2).
\eneq

To a morphism $f\cl M\to N$ of manifolds one associates the maps
\eq\label{eq:lagrcorresp}
&&\xymatrix{
TM\ar[rd]_-{\tau}\ar[r]^-{f'}&M\times_NTN\ar[d]^-{\tau}\ar[r]^-{f_\tau}&TN\ar[d]^-{\tau}\\
&M\ar[r]^-f&N,
}\quad\quad
\xymatrix{
T^*M\ar[rd]_-{\pi}&M\times_NT^*N\ar[d]^-{\pi}\ar[l]_-{f_d}\ar[r]^-{f_\pi}&T^*N\ar[d]^-{\pi}\\
&M\ar[r]^-f&N.
}\eneq
We set
\eq\label{eq:Tf}
&&Tf\eqdot f_\tau\circ f'\cl TM\to TN,
\eneq
and call $Tf$ the tangent map to $f$.

We denote by $\Gamma_f$ the graph of $f$.
Then, after identifying $T\Gamma_f$ with its image in $T(M\times N)$, we have $T\Gamma_f=\Gamma_{Tf}$.
We set
\eq\label{eq:lambdaf1}
&&\Lambda_f\eqdot T^*_{\Gamma_f}(M\times N)=(T\Gamma_f)^\perp.
\eneq
Then we have a commutative diagram in which $p_1$ and $p_2$ are induced by the projections $T^*(M\times N)$ to $T^*M$ and $T^*N$ and $p_2^a$ is the composition of $p_2$ and the antipodal map of $T^*N$:
\eq\label{eq:lambdaf2}
&&\xymatrix{
&\Lambda_f\ar[d]^\sim\ar[ld]_-{p_1}\ar[rd]^-{p^a_2}&\\
T^*M&M\times_NT^*N\ar[l]_-{f_d}\ar[r]^-{f_\pi}&T^*N.
}\eneq
We also set
\eq
&&T^*_MN\eqdot\opb{f_d}T^*_MM\simeq (T^*_MM\times T^*N)\cap\Lambda_f.\label{eq:TMN}
\eneq
and call $T^*_MN$ the conormal bundle to $M$ in $N$.

\subsubsection*{Quadratic forms}
 Let $\BBV$ be a real finite dimensional vector space, $\BBV^\C$ its complexification.
Let $Q$ be a quadratic form on  $\BBV$. We keep the same notation $Q$ 
to denote  the quadratic form defined on $\BBV^\C$.
We  set $Q_{>0} \eqdot \{ v \in \BBV ; Q(v) >0 \}$ and similarly with $Q_{\geq 0}$. 
We denote by $ \langle \cdot,\cdot \rangle_Q $ the bilinear form associated with $Q$.

\begin{lemma}\label{lem:quadra0}
Assume that  $Q$ has exactly one positive eigenvalue on $\BBV$.
\banum
\item
If $u, v \in \BBV$ and $Q(u) > 0$, then $Q(u + \sq v) \neq 0$.
\item
One has $\ol{Q_{>0}} = Q_{\geq 0}$.
\item
Let $\gamma$ be a connected component of $Q_{>0}$.
Then
$\ol{\gamma} = \{ v \in \BBV ; \langle v, w \rangle_Q \geq 0 \text{ for all } w \in \gamma\}$.
\eanum
\end{lemma}
The proof is left as an exercise.

Now let $Q$ be a quadratic form on a manifold $M$ and  let $\langle \scdot, \scdot \rangle_Q$ be the associated bilinear form on $TM$.
If $Q$ is nondegenerate, it induces an isomorphism $\sharp \colon TM \to T^*M, v \mapsto \langle v, \scdot \rangle_Q$, with inverse denoted here by $\flat$.
Therefore, there is an induced quadratic form, denoted by $Q_x^\sharp$, on each $T^*_xM$.
As usual, we shall write $v^\sharp=\sharp(v)$ for $v\in TM$.
We set $Q_{>0} \eqdot \{ v \in TM ; Q(v) >0 \}$ and similarly for $Q_{\geq 0}$.

\begin{lemma}\label{lem:quadra}
Let $M$ be a connected  manifold and $Q$ a continuous quadratic form on $M$ with exactly one positive eigenvalue.
Then 
\banum
\item
$Q_{>0}$ has at most two connected components and  $\pcl{Q_{>0}} = \ol{Q_{>0}} = Q_{\geq 0}$.
\item
Suppose that $Q_{>0}$ has two connected components and let $\gamma$ be one of them.
Then
\bnum
\item
$\gamma$ is an open convex proper cone and $\gamma_x \neq \varnothing$ for any $x \in M$,
\item
One has $\ol{\gamma} = \pcl{\gamma}$ and $\gamma = \Int( \ol\gamma )$.
\item
Suppose furthermore that $Q$ is nondegenerate.
Then $\gamma^\circ = \ol{\gamma}^\sharp$ and $\Int(\gamma^\circ) = \gamma^\sharp$.
\enum
\eanum
\end{lemma}
The proof is left as an exercise.

\subsubsection*{Preorders}
Consider a preorder $\preceq$ on a manifold $M$ and its graph $\gPo\subset M\times M$.
Then
\eqn
&&\Delta\subset\gPo\, ,\\
&&\gPo\circ\gPo=\gPo\,.
\eneqn
In the sequel, we shall often identify $\preceq$ and $\gPo$, that is, we shall call $\gPo$ ``a preorder''.
We denote by $\preceq^a$ the opposite preorder.

For a subset $A\subset M$, one sets
\eq\label{eq:J+}
&&\left\{
\parbox{70ex}{
$\paso A = q_1(\opb{q_2}(A)\cap \gPo)=\{x\in M;\mbox{ there exists } y\in A\mbox{ with } x\preceq y\}$,\\
$\futo A = q_2(\opb{q_1}(A)\cap \gPo)=\{x\in M;\mbox{ there exists } y\in A\mbox{ with } y\preceq x\}$.
}\right.
\eneq
For $x\in M$, we write $\futo x$ and $\paso x$ instead of $\futo{\{x\}}$ and $\paso{\{x\}}$ respectively.
One calls $\paso A$ (resp.\ $\futo A$) the past (resp.\ future) of $A$ for the preorder $\preceq$.

The next results are obvious:
\begin{itemize}
\item
$\paso{A}=\bigcup_{x\in A}\paso{x}$, and similarly with $\futo A$,
\item
$A\subset\paso A$, $\paso{\paso A}=\paso A$ and similarly with $\futo A$,
\item
$\gPoa=v(\gPo)$ where $v\cl M\times M\to M\times M$ is the map $v(x,y)=(y,x)$,
\item
$\paso{A}=\futoa{A}$,
\item
$A=\futo{A}\Leftrightarrow M\setminus A=\paso{M\setminus A}$.
(Indeed, assume $A=\futo{A}$ and let $x\notin A$, $y\in\pas{x}$.
Then $x\in\futo{y}$ which shows that $y\notin A$.)
\end{itemize}

\begin{definition}\label{def:regcone}
Let $\preceq$ be a preorder on $M$.
\banum
\item
The preorder is \define{closed} if $\gPo$ is closed in $M\times M$.
\item
The preorder is \define{proper} if $q_{13}$ is proper on $\gPo\times_M\gPo$.
\eanum
\end{definition}
In other words, a preorder $\preceq$ is proper if for any two compact subsets $A$ and $B$ of $M$, the so-called \define{causal diamond} $\futo A \cap \paso B$ is compact.

\begin{proposition}\label{pro:proper}
Let $\preceq$ be a preorder on $M$.
\bnum
\item
If $\preceq$ is closed and $A$ is a compact subset of $M$, then $\paso{A}$ and $\futo{A}$ are closed.
\item
If $\preceq$ is proper, then it is closed.
\enum
\end{proposition}

\begin{proof}
(i)
One has $\futo{A} = q_2 ( \gPo \cap q_1^{-1}(A) ) = q_2 ( \gPo \cap (A \times M) )$.
The projection $q_2$ is proper on $A \times M$, hence closed, and $\gPo \cap (A \times M)$ is closed, therefore $\futo{A}$ is closed.
The proof for $\paso{A}$ is similar.

\spa
(ii)
If the map $q_{13}\cl\gPo\times_M\gPo\to M\times M$ is proper, then it is closed.
Therefore, its image $\gPo$ is closed.
\end{proof}

\subsection{Causal manifolds}

In the literature, one often encounters time-orientable Lorentzian manifolds, to which one can associate a cone in $TM$ (see Definition~\ref{def:spacetime}).
Here, we only assume to be given a nowhere empty open convex cone $\gamma\subset TM$.

Recall that for a morphism of manifolds $f\cl M\to N$, the tangent map $Tf$ is defined in~\eqref{eq:Tf}.

\begin{definition}\label{def:causalmnf}
\banum
\item
A \define{causal manifold} $(M,\gamma)$ is a nonempty connected manifold $M$ equipped with an open convex cone $\gamma\subset TM$ such that $\gamma_x\neq\varnothing$ for all $x\in M$.

\item
A \define{morphism of causal manifolds} $f\cl (M,\gamma_M)\to(N,\gamma_N)$ is a morphism of manifolds such that $Tf(\ol{\gamma_M})\subset\ol{\gamma_N}$.

\item
A morphism of causal manifolds $f$ is {\em strict} if $Tf(\gamma_M)\subset\gamma_N$.
\eanum
\end{definition}
The composition of causal morphisms (resp.\ strictly causal) morphisms is a causal (resp.\ strictly causal) morphism, so that causal manifolds and their causal (resp.\ strictly causal) morphisms form a category.

For $U$ a open subset of $M$, we denote by $\gamma_U$ the cone $U\times_M\gamma$ of $TU$.
Then $(U,\gamma_U)$ is a causal manifold and the embedding $U\into M$ induces a morphism of causal manifolds $(U,\gamma_U)\to(M,\gamma)$.

\begin{notation}\label{not:causalI1}
We will often denote by $I$ an open interval of $\R$, which we will implicitly assume to contain $[0,1]$.
We denote by $t$ a coordinate on $I$, by $(t;v)$ the coordinates on $TI$ and by $(t;\tau)$ the associated coordinates on $T^*I$.
To $I$ we associate the causal manifold $(I,I\times \R_{>0})$ that we simply denote by $(I,+)$.
\end{notation}

\begin{proposition}\label{pro:strict}
If $f \colon (M,\gamma_M) \to (N,\gamma_N)$ is a causal submersion and if $\Int\left( \ol{\gamma_N}\right) = \gamma_N$, then $f$ is strictly causal.
In particular, if $\tim\cl (M,\gamma)\to(I,+)$ is a submersive causal morphism, then $\tim$ is strictly causal.
\end{proposition}

\begin{proof}
Since $f$ is submersive, then $Tf$ is open, so $Tf({\gamma_M}) \subset \Int\left( \ol{\gamma_N}\right) = \gamma_N$.
\end{proof}

\begin{notation}\label{not:lambda}
In this paper, for a causal manifold $(M,\gamma)$ we set
\eq\label{eq:gl}
&&\lambda\eqdot\gamma^\circ
\eneq
and similarly with $\lambda_M, \lambda_N$, etc.
Note that, by Lemma~\ref{le:lambdaclosed}, $\lambda$ is a closed convex proper cone of $T^*M$ and $\gamma=\Int(\lambda^\circ)$.
Moreover, $\lambda\supset T^*_MM$.
\end{notation}
\begin{remark}
One has $\ol{(\gamma_x)}=\lambda_x^\circ \subset (\ol{\gamma})_x$ but as seen in Example~\ref{exa:cones}, the inclusion may be strict.
\end{remark}

\begin{proposition}\label{pro:causalmor1}
Let $(M,\gamma_M)$ and $(N,\gamma_N)$ be two causal manifolds and let $f\cl M\to N$ be a morphism of manifolds.
Then $f$ is a morphism of causal manifolds if and only if $\Lambda_f\aconv\lambda_N\subset\lambda_M$.
\end{proposition}

\begin{proof}
Notice first that $Tf(\ol{\gamma_M})\subset\ol{\gamma_N}$ if and only if $\ol{\gamma_M}\conv \Gamma_{Tf}\subset\ol{\gamma_N}$.
At each $(x,y)\in \Gamma_f$, the vector subspace $(\Lambda_f)_{(x,y)}\subset T_{(x,y)}^*(M\times N)$ is the orthogonal to the vector subspace $(\Gamma_{Tf})_{(x,y)}\subset T_{(x,y)}(M\times N)$.
Hence, setting $E_1=T_xM$, $E_2=T_yN$, $\Gamma=(\Gamma_{Tf})_{(x,y)}$, we are reduced to prove that for two real finite-dimensional vector spaces $E_1$ and $E_2$, two closed convex cones $\gamma_1\subset E_1$, $\gamma_2\subset E_2$ and a linear graph $\Gamma\subset E_1\times E_2$ one has
\eqn
&&\gamma_1\conv \Gamma\subset \gamma_2\Leftrightarrow \Gamma^\perp\aconv\gamma_2^\circ\subset \gamma_1^\circ.
\eneqn
By hypothesis, $\Gamma$ is the graph of a linear map $u\cl E_1\to E_2$.
Therefore $\Gamma^\perp$ is the graph of the opposite transposed map $-{}^tu\cl E_2^*\to E_1^*$ and the result is clear since
\eqn
&& u(\gamma_1)\subset\gamma_2^{\circ\circ}\Leftrightarrow {}^t u(\gamma_2^\circ)\subset\gamma_1^\circ
\eneqn
and $\gamma_2 = \gamma_2^{\circ\circ}$ (since $\gamma_2$ is closed and convex).
\end{proof}

\begin{definition}\label{def:ctcone}
(i)
A \define{constant cone} contained in $\gamma$ is a triple $(\phi,U, \theta)$ where $\phi \colon U \to \R^d$ is a chart and $\theta \subset \R^d$ is an open convex cone, such that in this chart, $U \times \theta \subset \gamma$ (that is, $\phi(U) \times \theta \subset T\phi(\gamma|_U)$).
A constant cone $(\phi,U, \theta)$ will often be denoted simply by $U \times \theta$.

\spa
(ii)
A \define{basis of constant cones} contained in $\gamma$ is a family of constant cones whose union is $\gamma$.
\end{definition}

Although they are obvious, we state the two next lemmas which will be of frequent use.

\begin{lemma}\label{le:local}
Let $(M,\lambda)$ be a causal manifold.
Then there exists a basis of constant cones contained in $\gamma$.
\end{lemma}

\begin{proof}
It follows immediately from the fact that $\gamma\subset TM$ is open.
\end{proof}

\begin{lemma}\label{le:localbasis}
Let $(M,\lambda)$ be a causal manifold and let $A\subset M$.
Then $\gamma\subset N(A)$ if and only if there exists a basis of constant cones contained in $\gamma$ such that for $U\times\theta$ belonging to this basis, $U\cap(U\cap A+\theta)\subset A$.
\end{lemma}

\begin{proof}
This is a reformulation of~\eqref{eq:strictNC}.
\end{proof}

\begin{example}\label{exa:Lorentz1}
For the classical notions of Lorentzian manifold, spacetime and globally hyperbolic spacetime, references are made to~\cites{BGP07,BEE96,HE73,MS08}.

A \define{Lorentzian manifold} $(M,g)$ is a {\em connected} $C^\infty$-manifold $M$ with a $C^\infty$ nondegenerate bilinear form $g$ on $M$ of signature $(+,-,\dots,-)$.
Let
\eqn
&&g_{>0}=\{(x;v)\in TM; g_x(v,v)>0\}.
\eneqn
By Lemma~\ref{lem:quadra}, $g_{>0}$ has at most two connected components.
The Lorentzian manifold $(M,g)$ is \define{time-orientable} if the cone $g_{>0}$ has itself two connected components.
It is \define{time-oriented} if furthermore one connected component has been chosen.
\end{example}

\begin{definition}\label{def:spacetime}
A \define{Lorentzian spacetime} is a connected time-oriented Lorentzian manifold.
\end{definition}
Let $(M,g)$ be a Lorentzian spacetime.
We denote by $(M,\gamma_g)$, or $(M,\gamma)$ if there is no risk of confusion, the associated causal manifold.

\subsection{$\gamma$-sets and $\gamma$-topology}

The definition of the normal cone $N(A)$ to a subset $A$ as well as of the cone $D(A)$ and their main properties are recalled in the appendix.
Recall that $N(A)$ is an open convex cone.

\begin{definition}\label{def:gammaset}
Let $(M,\gamma)$ be a causal manifold.
A subset $A\subset M$ is a \define{$\gamma$-set} if $\gamma \subset N(A)$.
\end{definition}
Applying Lemma~\ref{le:localbasis}, we get:
\eq\label{eq:strictNC3}
&&\mbox{$A$ is a $\gamma$-set}\Leftrightarrow\left\{\parbox{50ex}{there exists a basis of constant cones $U\times\theta$ contained in $\gamma$ such that $U\cap(U\cap A+\theta)\subset A$.
}\right.
\eneq

\begin{proposition}\label{pro:gammaset1}
Let $(M,\gamma)$ be a causal manifold.
\bnum
\item
A set $A$ is a $\gamma$-set if and only if $\gamma_x\subset N_x(A)$ for all $x\in\partial A$.
\item
A subset $A$ of $M$ is a $\gamma$-set if and only if $M\setminus A$ is a $\gamma^a$-set.
\item
Let $(M_i,\gamma_i)$ \lp$i=1,2$\rp\, be causal manifolds.
Assume that $A_i$ is a $\gamma_{i}$-set for $i=1,2$.
Then $A_1\times A_2$ is a $(\gamma_{1}\times\gamma_2)$-set.
\item
Let $\gamma_1\subset\gamma_2$ be two open convex cones in $TM$.
If a set $A$ is a $\gamma_2$-set, then it is a $\gamma_1$-set.
\enum
\end{proposition}

\begin{proof}
(i)
This follows from Proposition~\ref{prop:cones}~(iv).

\spa
(ii)
This follows from Propositon~\ref{prop:cones}~(ii).

\spa
(iii)
This follows immediately from Proposition~\ref{prop:cones}~(x).

\spa
(iv)
Indeed, $\gamma_2 \subset N(A)$ implies $\gamma_1 \subset N(A)$.
\end{proof}

\begin{remark}
The converse to Proposition~\ref{pro:gammaset}~(iii) is true (see the remark following Proposition~\ref{prop:cones}) but we do not need it.
\end{remark}

\begin{proposition}\label{pro:gammaset}
Let $(M,\gamma)$ be a causal manifold.
\bnum
\item
The family of $\gamma$-sets is closed under arbitrary unions and intersections.
\item
The family of $\gamma$-sets is closed under taking closure and interior.
\item
If $A$ is a $\gamma$-set, then $\clos{\Int{A}} = \clos{A}$ and $\Int{\clos{A}} = \Int{A}$.
\item
If $A$ is a $\gamma$-set and $\Int{A} \subset B \subset \ol{A}$, then $B$ is a $\gamma$-set.
\enum
\end{proposition}

\begin{proof}
We shall use Notation~\eqref{eq:gl}.

\spa
(i)
Let $\{A_i\}_{i \in I}$ be a family of $\gamma$-sets, let $A = \bigcup_{i \in I} A_i$ and let $(x,v) \in D(A)$.
We shall prove that $(x,v)\notin\gamma$.

We choose a chart at $x$.
There is a sequence $\{(x_n,y_n,c_n)\}_n$ in $A\times(M\setminus A)\times\R_{>0}$ such that $x_n, y_n \to[n] x$ and $v_n = c_n (y_n - x_n) \to[n]v$.
Choose a function $\rho \cl \N \to I$ such that $x_n \in A_{\rho(n)}$.
Define
\eqn
&&t_n = \sup \{ t \in [0,1] \mid [x_n, x_n + t (y_n - x_n)] \subset A_{\rho(n)} \},\\
&&z_n = x_n + t_n (y_n - x_n),\\
&&x_{m,n} =
\begin{cases}
x_n + (t_n - 1/m)(y_n - x_n) \; (\text{for } m > 1/t_n) &\text{if } t_n\neq0,\\
x_n &\text{if } t_n = 0.
\end{cases}
\eneqn
Then $x_{m,n}\in A_{\rho(n)}$.

For all $n \in \N$, since $t_n$ is a supremum, there exists a sequence
$\{\delta_{m,n}\}_m$ in $\R_{\geq0}$ with $\delta_{m,n} \to[m] 0$ such that
\eqn
&&y_{m,n} = x_n + (t_n + \delta_{m,n})(y_n - x_n) \notin A_{\rho(n)}.
\eneqn
We also have $x_{m,n}, y_{m,n} \to[m] z_n$.

If $t_n = 0$, then $\delta_{m,n} >0$, so we can define
\eqn
&&c_{m,n} =
\begin{cases}
c_n/(1/m + \delta_{m,n}) &\text{if } t_n\neq0,\\
c_n/\delta_{m,n}&\text{if } t_n=0.
\end{cases}
\eneqn
Then $c_{m,n} >0$ and $c_{m,n}(y_{m,n} - x_{m,n}) = v_n$.
This proves that $(z_n, v_n) \in D(A_{\rho(n)}) \subset TM \setminus \gamma$.
Since $\gamma$ is open and $(z_n, v_n) \to[n] (x,v)$, this implies that $(x,v) \notin\gamma$.
Therefore $A$ is a $\gamma$-set.

\spa
The case of an intersection is deduced from (i) by Proposition~\ref{pro:gammaset1}~(ii).

\spa
(ii) follows immediately from Proposition~\ref{prop:cones}~(vi).

\spa
(iii) follows immediately from Proposition~\ref{prop:cones}~(vii).

\spa
(iv) The hypothesis and (iii) imply that $\Int{B} = \Int{A}$ and $\ol{B} = \ol{A}$.
It then follows from Proposition~\ref{prop:cones}~(viii) that $N(A)=N(B)$ and $B$ is a $\gamma$-set.
\end{proof}

\subsubsection*{$\gamma$-sets in vector spaces}

In vector spaces endowed with constant cones, $\gamma$-sets are easy to characterize.

Let $\BBV$ be a real finite-dimensional vector space, let $\Omega$ be a nonempty convex open subset of $\BBV$ and let $\gamma_0$ be an open convex cone in $\BBV$.
Set $\gamma=\Omega \times \gamma_0\subset T\Omega$, so that $(\Omega, \gamma)$ is a causal manifold.

\begin{proposition}\label{pro:lgtop}
A subset $A$ of $\Omega\subset\BBV$ is a $\gamma$-set if and only if
\eq\label{eq:lgtop}
\Omega \cap (A + \gamma_0) \subset A.
\eneq
\end{proposition}

\begin{proof}
(i)
Suppose that $A$ satisfies ~\eqref{eq:lgtop}.
Then, for each open subset $W\subset\Omega$, we get
\eqn
&&W \cap (A\cap W + \gamma_0) \subset A.
\eneqn
It follows from~\eqref{eq:strictNC3} that $A$ is a $\gamma$-set.

\spa
(ii)
Conversely, let $A \subset \Omega$ be a $\gamma$-set.
Let $x \in A$ and let $v \in \gamma_0$ with $x + v \in \Omega$.
We shall prove that $x + v \in A$.
Define $t_\infty = \sup \{ t \in[0,1]; [x, x + tv] \subset A\}$ and let $x_\infty = x + t_\infty v$.
Then $x_\infty\in\Omega$ and $(x_\infty,v) \in \gamma\subset N(A)$.
Therefore, there is a neighborhood $W \subset \Omega$ of $x_\infty$ such that $W \cap (A \cap W + \R_{>0}v) \subset A$.
Since $W$ is a neighborhood of $x_\infty$, there exists $\eta >0$ such that $[x_\infty, x_\infty+\eta v] \subset A$, which is a contradiction unless $t_\infty=1$.
\end{proof}

\subsubsection*{$\gamma$-topology}

Proposition~\ref{pro:gammaset} allows us to generalize~\cite{KS90}*{Def.~3.5.1}.

\begin{definition}\label{def:gammatop}
Let $(M,\gamma)$ be a causal manifold.
The \define{$\gamma$-topology} on $M$ is the topology for which the open sets are the open sets of $M$ which are $\gamma$-sets.
\end{definition}

A subset $A\subset M$ is called $\gamma$-open if it is open for the $\gamma$-topology.
In other words, if it is open in the usual topology and is a $\gamma$-set.
\begin{remark}
We shall not use the term $\gamma$-closed since a set which is closed for the $\gamma$-topology is not in general a $\gamma$-set, but is a $\gamma^a$-set.
\end{remark}
As in~\cite{KS90}*{Def.~3.5.1}, define the $\gamma_0$-topology on $\Omega$ by saying that an open set $U$ of $\BBV$ is $\gamma_0$-open if
\eq\label{eq:lgtop3}
&&U=\Omega\cap (U+\gamma_0).
\eneq

Applying Proposition~\ref{pro:lgtop}, we get that the $\gamma$-topology and the $\gamma_0$-topology on $\Omega$ coincide.

\subsection{The chronological preorder}

\begin{definition}\label{def:chronofut}
For $A\subset M$, we denote by $\futi{A}$ the intersection of all the $\gamma$-sets which contain $A$ and call it the \define{chronological future} of $A$.
We set $\futi{x}=\futi{\{x\}}$.
\end{definition}
Note that a set $A$ is a $\gamma$-set if and only if $\futi{A}=A$.

\begin{lemma}\label{le:gDel2}
The relation $y \in \futai{x}$ is a preorder.
\end{lemma}
\begin{proof}
Let $y\in\futai{x}$ and $z\in\futai{y}$.
Then $\futai{x}$ is a $\gamma$-set which contains $y$ and $\futai{y}$ is the smallest $\gamma$-set which contains $y$.
Therefore, $\futai{y}\subset\futai{x}$ and $z\in\futai{x}$.
\end{proof}

\begin{notation}
We denote by $\preceql$ the preorder given by $x\preceql y$ if $y\in\futai{x}$ and we denote by $\gDel$ the graph of this preorder.
Hence, using the notations~\eqref{eq:J+}, $\futai{x}=J^+_{\preceql}(x)$ and $\gDel=\Delta_{\preceql}$.
We call $\preceql$ the \define{chronological preorder}.
\end{notation}

\begin{remark}\label{not:causalI2}
Recall Notation~\ref{not:causalI1}.
On $(I,+)$ the chronological preorder $\preceql$ is the usual order $ \leq$.
\end{remark}

\begin{proposition}\label{pro:Iopen}
Let $A\subset M$ be a closed subset.
Then $\futai{A} \setminus A$ is open.
\end{proposition}

\begin{proof}
One has $\Int(\futai{A}) \subset \Int(\futai{A}) \cup A \subset \futai{A}$.
Applying Proposition~\ref{pro:gammaset}~(iv), we get that $\Int(\futai{A}) \cup A$ is a $\gamma$-set.
Since it contains $A$, it contains $\futai{A}$.
Therefore $\Int(\futai{A}) \cup A=\futai{A}$ and $\futai{A} \setminus A = \Int(\futai{A}) \setminus A$ is open.
\end{proof}

\begin{lemma}\label{le:local2}
Let $(M,\gamma)$ be a causal manifold and consider a constant cone $U \times \theta$ contained in $\gamma$.
Then, for $y, z \in U$ with $z-y \in \theta$, we have $z\in\futi{y}$.
\end{lemma}
\begin{proof}
Set $\gamma_1=U\times\theta$.
Then $z\in I^+_{\gamma_1}(y)$ by Proposition~\ref{pro:lgtop} and $I^+_{\gamma_1}(y)\subset U\cap\futi{y}$ by Proposition~\ref{pro:gammaset1}~(iv) (applied with $M=U$).
\end{proof}

\subsection{Causal paths}

\begin{notation}
If a function $c \colon I \to M$ is left (resp.\ right) differentiable, we denote its left (resp.\ right) derivative by $c'_l$ (resp.\ $c'_r$).
\end{notation}

\begin{definition}
A path $c \colon I \to M$ is a piecewise smooth map.
A path $c$ is \define{causal} if $c'_l(t), c'_r(t) \in (\ol\gamma)_{c(t)}$ for any $t \in I$ and it is {\em strictly causal} if $c'_l(t), c'_r(t) \in\gamma_{c(t)}$ for any $t \in I$.
\end{definition}

Hence, a smooth path $c$ is causal (resp.\ strictly causal) if and only if it defines a morphism (resp.\ strict morphism) of causal manifolds $c\cl (I,+)\to(M,\gamma)$.

Note that if $c_1$ and $c_2$ are two causal (resp.\ strictly causal) paths on $I$ with $c_1(1)=c_2(0)$, the concatenation $c=c_1\cup c_2$ (defined by glueing the two paths as usual) is causal (resp.\ strictly causal).

\begin{lemma}\label{le:morcausalpath}
Let $f\cl (M,\gamma_M)\to(N,\gamma_N)$ be a morphism \lp resp.\ a strict morphism\rp\, of causal manifolds and let $c\cl I\to M$ be a causal path \lp resp.\ a strictly causal path\rp.
Then $f\circ c\cl I\to N$ is a causal path \lp resp.\ a strictly causal path\rp.
\end{lemma}

\begin{proof}
This is a direct consequence of the chain rule.
\end{proof}

\begin{lemma}\label{le:tbcle0}
Let $c\cl I\to M$ be a strictly causal path.
Then for $t_1\leq t_2$ with $ t_1,t_2\in I$ we have $c(t_1)\preceql c(t_2)$.
\end{lemma}

\begin{proof}
It is enough to prove that for any $t_0\in I$, there exists $\alpha>0$ such that $c(t_0)\preceql c(t)$ for $t\in(t_0,t_0+\alpha)$ and similarly $c(t)\preceql c(t_0)$ for $t\in(t_0-\alpha,t_0)$.
We may assume $t_0=0$.
There exists a constant cone $U\times\theta$ contained in $\gamma$ and containing $(c(0),c'_r(0)$.
There exists $\alpha>0$ such that $c(t)-c(0)\in\theta$ for $t\in(0,\alpha)$.
By Lemma~\ref{le:local2}, this implies $c(0)\preceql c(t)$ for $t\in(0,\alpha)$.
The other case is similar, using $c'_l(0)$.
\end{proof}

\begin{lemma}\label{le:tbcle}
Let $A\subset M$.
One has $y\in\futai{A}$ if and only if $y\in A$ or there exists a strictly causal path $c\cl I\to M$ such that $c(0)\in A$, $c(1)=y$.
\end{lemma}

\begin{proof}
Let $B$ be the union of $A$ with the set of points that can be reached from $A$ by a strictly causal path.
We shall prove that $\futai{A}=B$.

\spa
(i)
To prove that $B \supset \futai{A}$, it is enough to check that $B$ is a $\gamma$-set.
Choose a constant cone $U \times \theta$ contained in $\gamma$ with $U$ convex.
By~\eqref{eq:strictNC3}, it is enough to prove that $U\cap(B\cap U+\theta)\subset B$.
Let $y' \in B \cap U$ and $v' \in \theta$ with $y'+v' \in U$.
Since $U$ is convex, $c \colon I \to U, t \mapsto y'+tv'$ is a strictly causal path for $I$ a small enough neighborhood of $[0,1]$.
Since $y' \in B$, there exists a strictly causal path $\tilde{c}$ with $\tilde{c}(0)\in A$ and $\tilde{c}(1)=y'$.
Therefore, concatenating $\tilde{c}$ and $c$ proves that $y'+v' \in B$.

\spa
(ii)
Let us prove that $B \subset \futai{A}$.
Let $y\in B, y\notin A$.
There exist $x\in A$ and a strictly causal curve $c$ going from $x$ to $y$.
Then $y\in \futai{x}$ by Lemma~\ref{le:tbcle0}.
Hence, $y\in\futai{A}$.
\end{proof}

\begin{remark}
Using Lemma~\ref{le:tbcle}, we obtain an alternate proof that the family of $\gamma$-sets is closed under unions and intersections.
Indeed, the proof of Lemma~\ref{le:tbcle} says that for any $A\subset M$, there is a smallest $\gamma$-set containing $A$, and it is the union of $A$ and the set of points that can be reached from $A$ by a strictly causal path.
Taking this as the definition of $\futai{A}$, for any set $A \subset M$, one has $\futai{A} = A$ if and only if $A$ is a $\gamma$-set.
Now, let $(A_i)_{i \in I}$ be a family of $\gamma$-sets and let $A = \bigcap_i A_i$.
Then $A = \bigcap_i \futai{A_i} = \bigcap_i \bigcup_{x \in A_i} \futai{x} = \bigcup_{x \in \bigcap_i A_i} \futai{x} = \bigcup_{x \in \bigcap_i A_i} \futai{x} = \futai{A}$, so $A$ is a $\gamma$-set.
The proof for unions is similar.
\end{remark}

\begin{example}
Let $M=\R^2$ with coordinates $(x_1,x_2)$ and let $(x_1,x_2;v_1,v_2)$ denote the coordinates on $TM$.
Consider the cones
\eqn
&&\theta_1=\{(v_1,v_2); v_2>|v_1|\},\quad \theta_2=\{(v_1,v_2); v_2>\frac{1}{2}|v_1|\}.
\eneqn
Let
\eqn
&&Z=\{(x_1,x_2);x_2=|x_1|\},\\
&&\gamma=(M\setminus Z)\times\theta_2\cup Z\times\theta_1.
\eneqn
Note that $\Int(\ol\gamma)=M\times\theta_2$.
One has $\futai{0}=\{0\}\cup \{(x_1,x_2);x_2>|x_1|\}$ and this set is strictly contained in $I^+_{\Int(\ol\gamma)}(0)=\{0\}\cup \{(x_1,x_2);x_2>\frac{1}{2}|x_1|\}$.
In other words, $\gamma$ and $\Int(\ol\gamma)$ define different chronological preorders.
\end{example}

\subsection{Causal preorders}

\begin{definition}\label{def:causalpo}
A preorder $\preceq$ is \define{causal} if $\gPo$ is a $(\gamma^a\times\gamma)$-set.
\end{definition}

\begin{lemma}\label{le:gDel5}
Let $\preceq$ be a preorder on $M$.
Assume that for any $x\in M$, $\futo{x}$ is a $\gamma$-set.
Then, for any $y\in M$, $\paso{y}$ is a $\gamma^a$-set.
\end{lemma}

\begin{proof}
For any $x$, $\futo{x}$ is a $\gamma$-set containing $x$.
Therefore, $\futai{x}\subset\futo{x}$ and $x\preceql y$ implies $x\preceq y$.
Let $y\in M$ and let $(x,v) \in \gamma$.
By Lemma~\ref{le:local}, there is a constant cone $U \times \theta \subset \gamma$ containing $(x,v)$.
We shall prove that $U\cap(U\cap\paso{y}-\theta)\subset\paso{y}$.
Let $x' \in U$ and $v' \in \theta$ be such that $x'-v' \in U$ and $x'\preceq y$.
We have $x'-(x'-v')=v'\in\theta$ so $x'-v' \preceql x'$ by Lemma~\ref{le:local2}.
Hence, $x'-v' \preceq x'$.
By hypothesis, $x' \preceq y$ so by transitivity, $x'-v' \preceq y$, that is, $x'-v' \in\paso{y}$.
Therefore, $(x,-v) \in N(\paso{y})$.
\end{proof}

\begin{lemma}\label{le:gDel1}
Let $\preceq$ be a preorder on $M$.
Assume that for any $x\in M$, $\futo{x}$ is a $\gamma$-set.
Then the preorder $\preceq$ is causal.
\end{lemma}

\begin{proof}
By the hypothesis, Lemma~\ref{le:gDel5} and Proposition~\ref{pro:gammaset1}~(iii), $\paso{x}\times\futo{x}$ is a $(\gamma^a\times\gamma)$-set.
Then the result follows from Proposition~\ref{pro:gammaset}~(i) and the equality
\eqn
&&\gPo=\bigcup_{x\in M}\paso{x}\times\futo{x}.
\eneqn
\end{proof}

\begin{lemma}\label{le:gDellambda}
One has $\gDel = I^+_{\gamma^a \times \gamma}(\Delta)$.
\end{lemma}

\begin{proof}
(i)
Lemma~\ref{le:gDel1} implies that $\gDel$ is a $(\gamma^a \times \gamma)$-set which proves the inclusion ``$\supset$''.

\spa
(ii)
For the reverse inclusion, let $(x,y) \in \gDel$.
By Lemma~\ref{le:tbcle}, there exists a strictly causal path $c \colon I \to M$ with $c(0)=x$ and $c(1)=y$.
The path $\tw{c} = (c_1,c_2) \colon I \to M \times M$ defined by $c_1(t) = c(1-t)$ and $c_2(t)=c(t)$ is a strictly causal path (for the causal structure on $M\times M$ given by $\gamma^a\times\gamma$) with $\tw{c}(1/2)=(c(1/2),c(1/2)) \in \Delta$ and $\tw{c}(1) = (x,y)$.
Therefore, again by Lemma~\ref{le:tbcle}, $(x,y) \in I^+_{\gamma^a \times \gamma}(\Delta)$.
\end{proof}

\begin{remark}
Since $\Delta$ is closed, Lemmas~\ref{pro:Iopen} and~\ref{le:gDellambda} imply that $\gDel\setminus\Delta$ is open.
\end{remark}

\begin{theorem}\label{th:gDel1}
Let $(M,\gamma)$ be a causal manifold and $\preceq$ be a preorder on $M$.
Then the following assertions are equivalent:
\bnum
\item
The preorder $\preceq$ is causal.
\item
For any $x\in M$, $\futo{x}$ is a $\gamma$-set.
\item
For any $y\in M$, $\paso{y}$ is a $\gamma^a$-set.
\item
For any $x\in M$, $\futi{x} \subset \futo{x}$.
\item
One has $\gDel \subset \gPo$.
\enum
\end{theorem}

\begin{proof}
(ii)$\Rightarrow$(i).
It is proved in Lemma~\ref{le:gDel1}.

\spa
(ii)$\Leftrightarrow$(iii).
By Lemma~\ref{le:gDel5}.

\spa
(i)$\Rightarrow$(v).
By Lemma~\ref{le:gDellambda}, $\gDel$ is the smallest $(\gamma^a\times\gamma)$-set containing the diagonal.
By hypothesis, $\gPo$ is a $\gamma^a\times\gamma$-set and contains the diagonal.
The result follows.

\spa
(iv)$\Leftrightarrow$(v).
Obvious.

\spa
(iv)$\Rightarrow$(ii).
We shall apply~\eqref{eq:strictNC3}.
Let $U \times \theta$ be a constant cone contained in $\gamma$ and let us prove that $U\cap(U\cap\futo{x}+\theta)\subset\futo{x}$.
Let $y' \in U\cap\futo{x}$ and let $v' \in \theta$ be such that $y'+v' \in U$.
By Lemma~\ref{le:local2}, $y' \preceql y'+v'$.
Since $\gDel\subset\gPo$, we obtain $y' \preceq y'+v'$.
Since $x \preceq y'$, we get $y'+v' \in \futo{x}$.
\end{proof}

Graphs of transitive relations, closed sets, and $\gamma$-sets in a causal manifold, are all closed under intersections.
This justifies Item~(a) of the following definition.

\begin{definition}\label{def:causalorder}
\banum
\item
One denotes by $\preceqccc$ the finest closed causal preorder and by $\gDelccc$ its graph, that is, $\gDelccc$ is the intersection of all graphs of closed causal preorders.
One calls it the canonical closed causal preorder, cc-preorder for short.
One denotes by $\futccc{A}$ and $\pasccc{A}$ the future and past sets of $A$ for the cc-preorder.
\item
One denotes by $\preceqps$ the preorder given by $x\preceqps y$ if there exists a causal path $c\cl I\to M$ with $c(0)=x$ and $c(1)=y$ and calls it the piecewise smooth preorder, ps-preorder for short.
One denotes by $\gDelps$ its graph and one denotes by $\futps{A}$ and $\pasps{A}$ the future and past sets of $A$ for the ps-preorder.
\item
If there is a risk of confusion, we denote by $\gDel^M$ the chronological preorder on $M$, and similarly for $\gDelps^M$ and $\gDelccc^M$.
\eanum
\end{definition}
The cc-preorder was introduced first for Lorentzian spacetimes in~\cite{SW96}.
One has
\eq\label{eq:gDelgDelccc}
&&\gDel \subset \gDelccc\mbox{ and }\gDel \subset \gDelps.
\eneq
The first inclusion is obvious by construction and the second one follows from Lemma~\ref{le:tbcle}.
Applying Theorem~\ref{th:gDel1}~(v)$\Rightarrow$(i), one gets:

\begin{corollary}
The preorder $\preceqps$ is causal.
\end{corollary}

\begin{proposition}\label{pro:increas}
Let $f \colon (M,\gamma_M) \to (N,\gamma_N)$ be a causal morphism.
\banum
\item
The function $f$ is increasing as a function from $(M,\preceqps)$ to $(N,\preceqps)$.
\item
If $f$ is strictly causal, then it is increasing as a function from $(M,\preceql)$ to $(N,\preceql)$.
\item
If either $\gDelps^N \subset \gDelccc^N$ or if $f$ is strictly causal, then $f$ is increasing as a function from $(M,\preceqccc)$ to $(N,\preceqccc)$.
\eanum
\end{proposition}

\begin{proof}
(a)
follows from Lemma~\ref{le:tbcle0} and the definition of $\preceqps$.

\spa
(b)
follows from Lemmas~\ref{le:tbcle0} and~\ref{le:morcausalpath}.

\spa
(c)
Let $A = \{ (x,y) \in M^2; f(x) \preceqccc f(y) \}$.
We shall prove that $A\supset\gDelccc$.
For that purpose, using Theorem~\ref{th:gDel1}, it is enough to check that $A$ is the graph of a closed preorder and $A\supset\gDel$.

\spa
(c)-(i)
$A$ is clearly the graph of a preorder and since $A = (\opb{f} \times \opb{f})(\gDelccc)$, it is closed.

\spa
(c)-(ii)
Let $x\preceql y$.
If $f$ is strictly causal then $f(x)\preceqccc f(y)$ by (b) and~\eqref{eq:gDelgDelccc}.
If $\gDelps^N \subset \gDelccc^N$, then $f(x)\preceqccc f(y)$ by (a) and~\eqref{eq:gDelgDelccc}.
Hence $\gDel\subset A$.
\end{proof}

\begin{proposition}
Let $(M,g)$ be a Lorentzian spacetime and let $(M,\gamma)$ be the associated causal manifold.
\banum
\item
One has
\eq
&&\gDel \subset \gDelps \subset \ol{\gDel} \subset \gDelccc.
\eneq
\item
The preoreder $\gDelps$ is a proper order if and only if the preorder $\gDelccc$ is a proper order and in this case, one has $\ol{\gDel} = \gDelps = \gDelccc$.
\eanum
\end{proposition}
One shall be aware that the inclusion $\ol{\gDel} \subset \gDelccc$ may be strict since the closure of a transitive relation need not be transitive, even in Lorentzian spacetimes.

\begin{proof}
(a)
The only inclusion left to prove is $\gDelps \subset \ol{\gDel}$.
It is classical, see for instance~\cite{MS08}*{Prop.~2.17}.

\spa
(b)
If $\gDelps$ is proper, then it is closed, so equal to $\preceqccc$ which is therefore a proper preorder.
For the converse and the second claim, see~\cite{MS08}*{Remark.~2.20}.
\end{proof}

We now extend the classical definition of global hyperbolicity of Lorentzian spacetimes to general causal manifolds as follows:

\begin{definition}\label{def:spacetime2}
A causal manifold $(M,\gamma)$ is \define{globally hyperbolic} if $\gDelccc$ is a proper order.
\end{definition}

\begin{example}
Let $M = \R^2 \setminus \{(1,0)\}$ and $\gamma = M \times (\R_{>0})^2$.
Then $(M,\gamma)$ is a causal manifold.
One easily checks that
\eqn
&&\futi{(0,0)} = \{(0,0)\} \cup (\R_{>0})^2,\\
&&\futps{(0,0)} = (\R_{\geq 0})^2 \setminus \big( [1,+\infty) \times \{0\} \big),\\
&&\futccc{(0,0)} = \ol{\futi{(0,0)}} = (\R_{\geq 0})^2 \setminus \{(1,0)\}.
\eneqn
In particular, $\futps{(0,0)}$ is neither closed nor open.
\end{example}

\subsection{Cauchy time functions and G-causal manifolds}\label{subsect:Gcausal}

The terminology G-causal below is not inspired by gravitation but by the name of Geroch.

\begin{definition}\label{def:Gcausal}
(a)
A \define{Cauchy time function} on a causal manifold $(M,\gamma)$ is a submersive causal morphism $q \colon (M,\gamma) \to (\R,+)$ which is proper on the sets $\futccc{K}$ and $\pasccc{K}$ for any compact set $K \subset M$.

\spa
(b)
A \define{G-causal manifold} $(M,\gamma,\tim)$ is the data of a causal manifold $(M,\gamma)$ together with a Cauchy time function $\tim$.
\end{definition}

\begin{proposition}\label{prop:timeIncreas}
A Cauchy time function on a causal manifold $(M,\gamma)$ is strictly causal and is increasing as a function from $(M,\preceqccc)$ to $(\R,\leq)$.
It is strictly increasing on strictly causal paths.
\end{proposition}

\begin{proof}
Both results follow from Propositions~\ref{pro:strict} and~\ref{pro:increas}, since $\preceqccc^{(\R,+)}$ is $\leq$.
\end{proof}

The above proposition implies that a causal manifold with a Cauchy time function cannot have strictly causal loops.
However, it may have causal loops, as Example~\ref{exa:loop2} shows.

In the definition of a Cauchy time function, it is enough to assume properness on the future and past of each point.

\begin{proposition}\label{pro:point}
Let $q \colon (M,\gamma) \to (\R,+)$ be a submersive causal morphism which is proper on the sets $\futccc{x}$ and $\pasccc{x}$ for any $x \in M$.
Then $q$ is a Cauchy time function on $(M,\gamma)$.
\end{proposition}

\begin{proof}
The proof is classical.
Let $K \subset M$ be a compact set and let $a,b\in \R, a\leq b$.
We shall prove that $q^{-1}([a,b]) \cap \pasccc{K}$ is compact, the case of $\futccc{K}$ being similar.

For any $x \in K$, let $y_x \in \futai{x} \setminus \{x\}$.
Then $x \in \pasai{y_x} \setminus \{y_x\}$, and this set being open, there is a compact neighborhood of $x$, say $V_x$, included in $\pasai{y_x} \setminus \{y_x\}$.
Since  $q^{-1}([a,b]) \cap \pasccc{V_x}$ is closed  (by Proposition~\ref{pro:proper}~(i)) and contained in the compact set $q^{-1}([a,b]]) \cap \pasccc{y_x}$, it is compact.

We cover $K$ with finitely many $V_x$'s, say the family $\{V_{x_i}\}_{i\in I}$. Then 
$q^{-1}([a,b]) \cap \pasccc{K}$ is closed and contained in the compact set $\bigcup_i \left( q^{-1}([a,b]) \cap \pasccc{V_{x_i}} \right)$.
Therefore,  it is compact.
\end{proof}

\begin{proposition}\label{pro:Gcausal3}
If a causal manifold admits a Cauchy time function, then its cc-preorder is proper.
\end{proposition}

\begin{proof}
Let $q$ be a Cauchy time function on a causal manifold $(M,\gamma)$.
Since $q$ is increasing, one has $q(\pasccc{x})\subset(-\infty,q(x)]$ and $\tim(\futccc{x}) \subset [q(x),+\infty)$ for any $x\in M$.
Let $K$ and $L$ be compact subsets of $M$.
There exist real numbers $a \leq b$ such that $\futccc{K} \cap \pasccc{L} \subset \opb{\tim}([a,b])$.
Since $\tim$ is proper on $\futccc{K}$, the set $\opb{\tim}([a,b]) \cap \futccc{K}$ is compact.
Recalling that $\pasccc{L}$ is closed, the set $\futccc{K}\cap\pasccc{L} = \opb{\tim}([a,b]) \cap \futccc{K} \cap \pasccc{L}$ is compact.
\end{proof}

\begin{proposition}\label{prop:timeSurj}
Let $q$ be a Cauchy time function on $(M,\gamma)$ and let $x \in M$.
Then $\futai{x}$ is not relatively compact and $q(\futai{x}) = q(\futccc{x}) = [q(x), +\infty)$.
In particular, G-causal manifolds cannot be compact and Cauchy time functions are surjective.
\end{proposition}

\begin{proof}
One has $q(\futai{x}) \subset q(J^+_{cc}(x)) \subset [q(x), +\infty)$ by Proposition~\ref{prop:timeIncreas}.
Since $\futai{x}$ is connected, the set $q(\futai{x})$ is an interval.
By properness of $q$ on $\futccc{x}$, hence on $\ol{\futai{x}}$, this interval is bounded if and only if $\futai{x}$ is relatively compact.
Suppose that this is the case.
Then $q \left( \ol{\futai{x}} \right) = [q(x),t_\infty]$ for some $t_\infty \in [q(x), +\infty)$.
Let $y \in \ol{\futai{x}}$ with $q(y) = t_\infty$, fix a chart $U$ at $y$, and let $v \in \gamma_y$.
The set $\ol{\futai{x}}$, being the closure of a $\gamma$-set, is also a $\gamma$ set.
Therefore, for $\epsilon>0$ small enough, the path $c \colon (-\epsilon,\epsilon) \to U, t \mapsto y + tv$ satisfies $c \big( [0,\epsilon) \big) \subset \ol{\futai{x}}$ and this shows that there exists $\eta >0$ such that $[q(x), t_\infty+\eta) \subset q \left( \ol{\futai{x}} \right)$, which contradicts our assumption.
\end{proof}

\begin{theorem}\label{th:GBS}
If a Lorentzian spacetime is globally hyperbolic, then it admits a Cauchy time function.
\end{theorem}

\begin{proof}
See~\cite{MS08}*{Thm.~3.75} and~\cite{Ge70}*{Prop.~8} (which suffices in view of Proposition~\ref{pro:point}).
See also~\cite{FS11} for a more general version.
\end{proof}

\section{Sheaves on causal manifolds}\label{section:gcausal}

\subsection{Microsupport}\label{subsection:SSi}
We recall here a few basic results on the microlocal theory of sheaves and refer to~\cite{KS90}.

For simplicity, we denote by $\cor$ a field, although all results would remain true when $\cor$ is a commutative unital ring of finite global dimension.

Let $\Derb(\cor_M)$ denote the bounded derived category of sheaves of $\cor$-modules on $M$.
For $F\in\Derb(\cor_M)$, its microsupport, or singular support, denoted by $\SSi(F)$, is a closed conic co-isotropic subset of $T^*M$ whose intersection with $T^*_MM$ is $\supp(F)$, the support of $F$ (see~\cite{KS90}*{Def.~5.1.2}).
Roughly speaking, the microsupport describes the set of codirections of non propagation.

For a locally closed subset $A$ of $M$, one denotes by $\cor_A$ the constant sheaf on $A$ with stalk $\cor$ extended by $0$ on $X\setminus A$.
For $F\in\Derb(\cor_M)$, one sets $F_A\eqdot F\tens\cor_A$.
Recall that if $Z$ is closed in $M$, then $\rsect(M;F_Z)\simeq\rsect(Z;F\vert_Z)$.

We shall also make use of the dualizing complex on $M$ denoted by $\omega_M$.
Recall that $\omega_M$ is isomorphic to the orientation sheaf shifted by the dimension of $M$.
It is an invertible sheaf and for a morphism $f\cl M\to N$, one has $\epb{f}\omega_N\simeq\omega_M$ and $\epb{f}\cor_N\simeq\omega_{M/N}=\omega_M\tens\opb{f}\omega_N^{\otimes-1}$.

We first recall a few results of constant use:

\begin{proposition}[see~\cite{KS90}*{Prop.~5.3.8}]\label{pro:KS538}
Let $Z,U\subset M$.
Assume that $Z$ is closed and $U$ is open.
Then $\SSi(\cor_Z)\subset N(Z)^\circ$ and $\SSi(\cor_U)\subset N(U)^{\circ a}$.
\end{proposition}

\begin{theorem}[\cite{KS90}*{Prop.~5.4.14}]\label{th:oper1}
Let $F_1,F_2\in\Derb(\cor_M)$.
\bnum
\item
Assume that $\SSi(F_1)\cap\SSi(F_2)^a\subset T^*_MM$.
Then $\SSi(F_1\tens F_2)\subset \SSi(F_1)+\SSi(F_2)$.
\item
Assume that $\SSi(F_1)\cap\SSi(F_2)\subset T^*_MM$.
Then $\SSi(\rhom(F_1,F_2))\subset \SSi(F_1)^a+\SSi(F_2)$.
\enum
\end{theorem}
Consider a morphism of manifolds $f\cl M\to N$ and recall the maps $f_d$ and $f_\pi$ in~\eqref{eq:lagrcorresp}.
\begin{definition}
Let $G\in\Derb(\cor_N)$.
One says that $f$ is non-characteristic for $G$ if $f_d$ is proper on $\opb{f_\pi}\SSi(G)$.
\end{definition}
Since $\opb{f_\pi}\SSi(G)$ is conic in $M\times_NT^*N$, this is equivalent to the condition
\eq
&&\opb{f_d}T^*_MM\cap\opb{f_\pi}\SSi(G)\subset M\times_NT^*_NN.
\eneq

\begin{theorem}[\cite{KS90}*{Prop.~5.4.4,~5.4.5}]\label{th:oper2}
Consider a morphism of manifolds $f\cl M\to N$.
\bnum
\item
Let $F\in \Derb(\cor_M)$ and assume that $f$ is proper on $\supp(F)$.
Then $\SSi(\roim{f}F)\subset f_\pi\opb{f_d}(\SSi(F))$.
If $f$ is a closed embedding then this inclusion is an equality.
\item
Let $G\in \Derb(\cor_{N})$ and assume that $f$ non-characteristic for $G$.
Then $\SSi(\opb{f}G)\subset f_d\opb{f_\pi}(\SSi(G))$.
If $f$ is a submersion then this inclusion is an equality.
\enum
\end{theorem}

Consider a morphism of manifolds $f\cl M\to N$ and let $F\in\Derb(\cor_M)$.
When $f$ is proper on $\supp(F)$, Theorem~\ref{th:oper2} gives a bound to the microsupport of $\roim{f}F$.
However, we shall have to consider a non proper situation.
The next lemma already appeared in~\cite{GS14} (with a slightly different formulation) but we give here another and more elementary proof.
It is a variation on~\cite{KS90}*{Exe.~V.7}.

\begin{lemma}\label{le:ssioimf}
Let $f\cl M\to N$ be a morphism of manifolds and let $F\in\Derb(\cor_M)$.
Assume to be given an increasing family $\{Z_n\}_n$ \lp$n\in\N$\rp\, of closed subsets of $M$ such that $Z_n\subset \Int(Z_{n+1})$ for all $n$ and $M=\bigcup_nZ_n$.
Then
\eq\label{eq:SSoim1}
&&\SSi(\roim{f}F)\subset \ol{\bigcup_n\SSi\bl\roim{f}(F_{Z_n})\br}.
\eneq
\end{lemma}

\begin{proof}
Let $p\in T^*N$ with
$p\notin\ol{\bigcup_n\SSi\bl\roim{f}(F_{Z_n})\br}$.

\spa
(i)
We may assume $p\notin T^*_NN$, otherwise the result is obvious.

\spa
(ii)
We may assume that $N$ is open in some vector space $\BBV$ and there exist an open neighborhood $V$ of $\pi(p)$ in $N$, an open cone $\lambda_0$ in $\BBV^*$ with $p\in V\times\lambda_0$ such that
\eqn
&&(V\times \lambda_0)\cap\SSi\bl\roim{f}(F_{Z_n})\br=\varnothing\mbox{ for all }n.
\eneqn
Let $\gamma_1$ be a closed convex proper cone contained in $(\lambda_0)^{\circ a}$ and let $\Omega_0\subset\Omega_1$ be two $\gamma_1$-open subsets of $\BBV$ with
\eqn
&&\Omega_1\setminus\Omega_0\subset N,\\
&&\mbox{for any $x\in\Omega_1, (x+\gamma_1)\setminus\Omega_0$ is compact.}
\eneqn
Applying~\cite{KS90}*{Prop.~5.2.1}, we obtain the isomorphism
\eq\label{eq:SSoim2}
&&\rsect(\Omega_1\cap N;\roim{f}(F_{Z_n}))\isoto\rsect(\Omega_0\cap N;\roim{f}(F_{Z_n})).
\eneq
or equivalently, the isomorphism
\eq\label{eq:SSoim3}
&&\rsect\bl\opb{f}(\Omega_1\cap N)\cap Z_n;F\br\isoto\rsect\bl\opb{f}(\Omega_0\cap N)\cap Z_n;F\br.
\eneq
Consider a distinguished triangle
\eqn
&&\rsect_{\opb{f}(\Omega_1\cap N)}F\to\rsect_{\opb{f}(\Omega_0\cap N)}F\to G\to[+1]
\eneqn
Then $\rsect(Z_n;G)\simeq0$ and by the Grothendieck Mittag-Leffler theorem (see for example~\cite{KS90}*{Prop.~2.7.1~(iii)}), we get $\rsect(X;G)\simeq0$, whence the isomorphism
\eq\label{eq:SSoim5}
&&\rsect(\Omega_1\cap N;\roim{f}F)\isoto\rsect(\Omega_0\cap N;\roim{f}F).
\eneq
It follows from the definition of the microsupport that $p\notin\SSi(\roim{f}F)$.
\end{proof}

The next result is well-known from the specialists, but, to our knowledge, is not in the literature.
Consider the Cartesian square of real manifolds
\eq\label{eq:cartsq1}
&&\xymatrix{
Y\ar[r]^-j\ar@{}[rd]|-\square\ar[d]_-{g}&X\ar[d]^-f\\
N\ar[r]_-i&M.
}\eneq
Let $F\in\Derb(\cor_X)$.
One has a natural isomorphism~\cite{KS90}*{Prop.~3.1.9}:
\eq\label{eq:epbbasech}
&&\epb{i}\roim{f}F\simeq\roim{g}\epb{j}F.
\eneq
\begin{lemma}[Non-characteristic base change formula]\label{le:ncbasech}
Let $F\in\Derb(\cor_X)$.
Assume that $j$ is non-characteristic for $F$, $i$ is non-characteristic for $\roim{f}F$ and $f$ is submersive.
Then the isomorphism~\eqref{eq:epbbasech} induces the isomorphism $\opb{i}\roim{f}F\simeq\roim{g}\opb{j}F$.
\end{lemma}

\begin{proof}
Using~\cite{KS90}*{Prop.~5.4.13} and the hypotheses that $i$ and $j$ are non-characteristic, we get the isomorphisms
\eqn
\epb{i}\roim{f}F&\simeq&\opb{i}\roim{f}F\tens\epb{i}\cor_M,\\
\roim{g}\epb{j}F&\simeq& \roim{g}(\opb{j}F\tens \epb{j}\cor_X).
\eneqn
On the other hand, we have the isomorphisms
\eqn
\roim{g}(\opb{j}F\tens \epb{j}\cor_X)&\simeq&\roim{g}(\opb{j}F\tens\opb{g}\epb{i}\cor_M)\\
&\simeq&\roim{g}\opb{j}F\tens\epb{i}\cor_M.
\eneqn
The isomorphism $\epb{j}\cor_X\simeq \opb{g}(\epb{i}\cor_M)$ follows from the hypothesis that $f$ is submersive.
The last isomorphism follows since locally on $N$, $\epb{i}\cor_M$ is free of finite rank (up to a shift).
Using~\eqref{eq:cartsq1}, we get the the isomorphism
\eqn
&&\opb{i}\roim{f}F\tens \epb{i}\cor_M\simeq\roim{g}\opb{j}F\tens \epb{i}\cor_M.
\eneqn
The result follows since $\epb{i}\cor_M$ is invertible.
\end{proof}

\subsection{Propagation and Cauchy problem}\label{subsect:propagcausal}

Recall Notation~\ref{not:lambda} in which  we set $\lambda=\gamma^\circ$.

\begin{proposition}\label{pro:SSgamcl}
Let $(M,\gamma)$ be a causal manifold.
Let $Z,U\subset M$.
Assume that $U$ is open and is a $\gamma$-set and that $Z$ is closed and is a $\gamma^a$-set.
Then $\SSi(\cor_U)\subset\lambda^{a}$ and $\SSi(\cor_Z)\subset\lambda^{a}$.
\end{proposition}

\begin{proof}
By hypothesis, $\gamma\subset N(U)$.
Hence, $\SSi(\cor_U)\subset N(U)^{\circ a}\subset\lambda^{a}$ by Proposition~\ref{pro:KS538}.

\spa
The same proof applies to $Z$.
\end{proof}

\begin{corollary}\label{cor:SSgamclbis}
Let $(M,\gamma)$ be a causal manifold and $\preceq$ be a closed causal preorder on $M$.
Let $Z,U\subset M$.
Assume that $U$ is open and $U=\futo{U}$ and assume that $Z$ is closed and $Z=\paso{Z}$.
Then $\SSi(\cor_U)\subset\lambda^{a}$ and $\SSi(\cor_Z)\subset\lambda^{a}$.
\end{corollary}

\begin{proof}
Apply Proposition~\ref{pro:SSgamcl} and Theorem~\ref{th:gDel1}, (i)$\Rightarrow$(ii).
\end{proof}

Recall that a morphism of manifolds $f\cl M\to N$ gives rise to the maps
\eqn
&&\xymatrix{
T^*M&\ar[l]_-{f_d}M\times_N T^*N\ar[r]^-{f_\pi}&T^*N.
}\eneqn

\begin{theorem}\label{th:GP1}
Let $f\cl (M,\gamma_M)\to (N,\gamma_N)$ be a morphism of causal manifolds, let $\preceq$ be a closed causal preorder on $M$ and let $F\in\Derb(\cor_M)$.
Assume that
\banum
\item
$f\cl M\to N$ is submersive,
\item
for any compact $K\subset M$, the map $f$ is proper on the closed set $\paso{K}$,
\item
$\SSi(F)\cap\lambda_M\subset T^*_MM$.
\eanum
Then
\eq\label{eq:ssoimtimF2}
&&\SSi(\roim{f}F)\cap\Int(\lambda_N)=\varnothing.
\eneq
\end{theorem}

\begin{proof}
(i)
Let $K$ be a compact subset of $M$ and let $Z = \paso{K}$.
Then $Z$ is closed  by Proposition~\ref{pro:proper}, $Z=\paso{Z}$ and $\SSi(\cor_{Z})\subset \lambda_M^{a}$ by Corollary~\ref{cor:SSgamclbis}.
Since the cone $\lambda_M$ is closed convex and proper, 
we obtain by  applying Proposition~\ref{th:oper1}~(i)
\eq\label{eq:ssiFn}
&&\SSi(F_Z)\cap\lambda_M\subset T^*_MM.
\eneq
Since $f$ is submersive (hence $f_d$ is injective), we deduce from~\eqref{eq:ssiFn}
\eqn
&&\opb{f_d}(\SSi(F_Z))\cap\opb{f_d}\lambda_M\subset M\times_NT^*_NN
\eneqn
and using the fact that $f$ is causal
\eqn
&&\opb{f_d}(\SSi(F_Z))\cap\opb{f_\pi}\lambda_N\subset M\times_NT^*_NN.
\eneqn
Since $f$ is proper on $\supp(F_Z)$, we get:
\eq\label{eq:ssiFn2}
&&\SSi(\roim{f}F_Z)\cap\lambda_N\subset f_\pi\opb{f_d}(\SSi(F_Z))\cap\lambda_N\subset T^*_NN.
\eneq

\spa
(ii)
Let $W$ be an open relatively compact subset of $N$ and set $V=\opb{f}W$, $f_V=f\vert_V$.
Since $f$ is proper on the sets $\paso{K}$ ($K$ compact in $M$), we may construct by induction an exhaustive sequence $\{K_n\}_n\in \N$ of compact subsets of $M$ such that
\eqn
&&\paso{K_n}\cap V\subset \Int (\paso{K_{n+1}}) \cap V.
\eneqn
Set $Z_n=\paso{K_n}$ and $F_n=F_{Z_n}$.
By~\eqref{eq:ssiFn2}
\eqn
&&\SSi(\roim{f}F_n)\cap T^*W\cap\Int(\lambda_N)=\varnothing.
\eneqn
Applying  Lemma~\ref{le:ssioimf} to the map $f_V\cl V\to W$, we get:
\eq\label{eq:ssoimtimF3}
&&\SSi(\roim{f}F)\cap T^*W\cap\Int(\lambda_N)=\varnothing.
\eneq
Since this result holds for any $W$ open relatively compact, the proof is complete.
\end{proof}

\begin{corollary}\label{cor:GP2}
Let $I$ be a finite set and let $\{\gamma_{M,i}\}_{i\in I}$ and $\{\gamma_{N,i}\}_{i\in I}$ be two families of open convex cones in $TM$ and $TN$ respectively.
Let $f\cl M\to N$ be a morphism of manifolds which defines morphisms of causal manifolds $f\cl(M,\gamma_{M,i})\to(N,\gamma_{N,i})$.
For each $i\in I$ let $\preceq_i$ be a closed preorder on $M$, causal for $\gamma_i$.
Let $F\in\Derb(\cor_M)$.
Assume
\banum
\item
$f$ is submersive,
\item
for any compact $K\subset M$ and all $i\in I$, the map $f$ is proper on $\pasii{K}$,
\item
$\SSi(F)\cap\lambda_{M,i}\subset T^*_MM$ for all $i\in I$,
\item
$\bigcup_i\Int(\lambda_{N,i})=T^*N\setminus T^*_NN$.
\eanum
Then
\eq\label{eq:ssoimtimF4}
&&\SSi(\roim{f}F)\subset T^*_NN.
\eneq
In other words, $\roim{f}F$ is a local system \lp in the derived sense\rp.
Moreover, $f$ is surjective.
\end{corollary}

\begin{proof}
The inclusion~\eqref{eq:ssoimtimF4} follows from Theorem~\ref{th:GP1}.
Since $M$ is nonempty, $\roim{f}\cor_M$ is a non-zero local system on $N$.
Since $N$ is connected, the result follows.
\end{proof}

\begin{corollary}\label{cor:GP3}
We make the same hypotheses as in Corollary~\ref{cor:GP2} and we assume moreover that $N$ is contractible.
For $a\in N$, set $M_a=\opb{f}(\{a\})$.
Then the restriction morphism
\eq
&&\rsect(M;F)\to\rsect(M_a;F\vert_{M_a})
\eneq
is an isomorphism.
\end{corollary}

\begin{proof}
It follows from Corollary~\ref{cor:GP2} that $\SSi(\roim{f}F)\subset T^*_NN$.
In other words, all cohomology objects of $\roim{f}F$ are local systems on $N$.
This last manifold being contractible, we get:
\eq\label{eq:oimtim1}
&&\rsect(N;\roim{f}F)\simeq(\roim{f}F)_a.
\eneq
Consider the Cartesian square
\eqn
&&\xymatrix{
M_a\ar[r]^-\iota\ar@{}[rd]|-\square\ar[d]_-{f_a}&M\ar[d]^-f\\
\{a\}\ar[r]_-j&N.
}\eneqn
It follows from the hypotheses that $\iota$ is non-characteristic for $F$ and it follows from Corollary~\ref{cor:GP2} that $j$ is non-characteristic for $\roim{f}F$.
Applying Lemma~\ref{le:ncbasech}, we get
\eq\label{eq:oimtim2}
&&\opb{j}\roim{f}F\simeq\roim{f_a}\opb{\iota}F.
\eneq
By using~\eqref{eq:oimtim1}, we get
\eqn
&&\rsect(M;F)\simeq \rsect(N;\roim{f}F)\simeq (\roim{f}F)_a\simeq\opb{j}\roim{f}F\simeq\roim{f_a}\opb{\iota}F\simeq\rsect(M_a;F\vert_{M_a}).
\eneqn
\end{proof}

\subsection{Sheaves on G-causal manifolds}\label{subsect:ShvGcausal}

We shall particularize the results of Subsection~\ref{subsect:propagcausal} to the case where $N = \R$, the interesting case in practice.

In the sequel, we denote by $t$ a coordinate on $\R$ and by $(t;\tau)$ the associated coordinates on $T^*\R$.
We shall write for short $\{\tau\geq0\}$ instead of $\{(t;\tau)\in T^*\R;\tau\geq0\}$ and similarly with $\tau\leq0$.

\begin{proposition}\label{pro:GP1}
Let $(M,\gamma,\tim)$ be a G-causal manifold and let $F\in\Derb(\cor_M)$.
Assume that $\SSi(F)\cap\lambda\subset T^*_MM$.
Then
\eq\label{eq:ssoimtimF1}
&&\SSi(\roim{\tim}F)\subset\{\tau\leq0\}.
\eneq
\end{proposition}

\begin{proof}
Apply Theorem~\ref{th:GP1} with $N=\R$, $f=\tim$, $\gamma_N=\{\tau\geq0\}$.
\end{proof}

\begin{theorem}\label{th:GglobalCP1}
Let $(M,\gamma,\tim)$ be a G-causal manifold and let $F\in\Derb(\cor_M)$.
\bnum
\item
Assume that $\SSi(F) \cap \lambda^{a} \subset T^*_MM$ and let $B$ be a closed subset satisfying $B = \paso{B}$ and $B \subset \opb{\tim}((-\infty,a])$ for some $a \in \R$.
Then
\eq\label{eq:sectAF0}
&&\rsect_B(M;F)\simeq0.
\eneq
\item
Assume that $\SSi(F)\cap(\lambda\cup\lambda^{a})\subset T^*_MM$.
Then, setting $M_0 = \opb{\tim}(0)$, the natural restriction morphism below is an isomorphism:
\eq\label{eq:GglobalCP}
&&\rsect(M;F)\isoto\rsect(M_0;F\vert_{M_0}).
\eneq
\enum
\end{theorem}

\begin{proof}
(i)
By Corollary~\ref{cor:SSgamclbis}, $\SSi(\cor_{B})\subset \lambda^{a}$.
Applying Proposition~\ref{th:oper1}~(ii), we get
\eqn
&&\SSi(\rsect_B(F))\subset \SSi(F)+\lambda.
\eneqn
Therefore, $\SSi(\rsect_BF) \cap \lambda^{a}\subset T^*_MM$ and we may apply Proposition~\ref{pro:GP1} to this sheaf (with $\lambda^{a}$ instead of $\lambda$).
We obtain
\eqn
&&\SSi(\roim{\tim}\rsect_BF) \subset\{\tau\geq0\}.
\eneqn
Since $\supp(\rsect_BF) \subset (-\infty,a]$ for some $a \in \R$, this implies
\eqn
&&\rsect_B(M;F)\simeq\rsect(\R;\roim{\tim}\rsect_BF)\simeq0.
\eneqn

\spa
(ii)
Apply Corollary~\ref{cor:GP3}.
\end{proof}

\begin{remark}
In the paper~\cite{DS98}, the notion of a $\lambda$-propagator is introduced.
Essentially, on a causal manifold $(M,\gamma)$, an object $K\in\Derb(\cor_{M\times M})$ is a $\lambda$-propagator if it satisfies:
\eq\label{eq:propag1}
&&\left\{
\parbox{70ex}{
(i) the identity morphism $\cor_M\to\cor_M$ factors through $\reim{q_2}K$,\\
(ii) $\SSi(K)\subset T^*M\times\lambda^a$, \\
(iii) $\SSi(K)\cap (T^*M\times T^*_MM)\subset T^*_{M\times M}(M\times M)$.
} \right.
\eneq
Denote by $\preceq$ any closed causal preorder on $M$.
Then one can prove along the lines of loc.\ cit.\ that
\eq\label{eq:propag}
&&\left\{\parbox{70ex}{
Assume that there exists a $\lambda$-propagator $K$.
Let $B\subset M$ be a closed subset such that $B$ is $\gamma$-proper, $B\neq M$ and $B=\paso{B}$.
Then for any $F\in\Derb(\cor_M)$ satisfying $\SSi(F)\cap\lambda^{a}\subset T^*_MM$, one has $\rsect_B(M;F)\simeq0$.
}\right.\eneq
In other words, one recovers the conclusions of Theorem~\ref{th:GglobalCP1} (indeed, one can also treat the Cauchy problem) when assuming the existence of a propagator instead of that of a Cauchy time function.

However, it seems difficult to construct $\lambda$-propagators, contrarily to what is written in~\cite{DS98}*{Prop.~4.4~(ii)} (see below).
\end{remark}

\begin{erratum}\label{err:2}
In~\cite{DS98}*{Prop.~4.4~(ii)}, it is asserted that under mild conditions on the preorder, the constant sheaf \lp or a variant of this sheaf\rp\, on the graph of the causal preorder is a propagator.
However, the proof is not complete and indeed, the result is not correct without extra hypotheses, as seen in Example~\ref{exa:DS} below.
Note that the results of~\cite{Sc13}*{\S~6}, being built on this wrong statement, they should be replaced with those of this paper.
\end{erratum}

\begin{example}\label{exa:DS}
Let $M = \R^2$ be the plane with coordinates $(x,t)$.
Define the open convex cones
\begin{align*}
\gamma_0^- &= \{ (v,w) \in \R^2 ; w > \max(0,-v) \} \\
\gamma_0^+ &= \{ (v,w) \in \R^2 ; w > |v| \}
\end{align*}
and define the open convex cone $\gamma \subset TM \simeq M \times \R^2$ by
\begin{equation*}
\gamma = (\R \times \R_{<0}) \times \gamma_0^- \cup M \times \gamma_0^+
\end{equation*}
so that $(M,\gamma)$ is a causal manifold.
The cc-preorder is given by
\begin{equation*}
\futccc{(x,t)} = (x,t) + \ol{\gamma_0^+} \text{ if $t>0$, and } (x,t) + \ol{\gamma_0^-} \text{ if $t \leq 0$}.
\end{equation*}
In particular, $(M,\gamma)$ is easily seen to be globally hyperbolic.
One also checks that if $0<\alpha<1$, then $(x,t) \mapsto t + \alpha x$ is a Cauchy time function on $(M,\gamma)$.

We shall prove that $((0,-1),(0,0)) \in (\SSi(\cor_{\gDelccc}))_{((0,0),(1,0))}$, which implies that $\cor_{\gDelccc}$ is not a propagator (it does not satisfy~\eqref{eq:propag1}~(iii)).
Define the open balls $U_1 = B((0,0);1/4)$ and $U_2 = B((1,0);1/4)$.
Then $U_1 \times U_2$ is a neighborhood of $((0,0),(1,0)) \in M^2$ and one has
\begin{equation*}
\gDelccc \cap (U_1 \times U_2) = \{ ((x_1,t_1), (x_2,t_2)) \in U_1 \times U_2 ; t_1 \leq \min(0,t_2) \}.
\end{equation*}
By~\cite{KS90}*{Prop.~5.3.1}, this implies that
\begin{equation*}
(\SSi(\cor_{\gDelccc}))_{((0,0),(1,0))} = \{ ((\xi_1,\tau_1), (\xi_2,\tau_2)) \in \R^2 \times \R^2 ; 0 \leq \tau_2 \leq -\tau_1 \}
\end{equation*}
which contains $((0,-1),(0,0))$.

Note that one can obtain a similar counter-example with a continuous cone, namely, on $M = \R^2$, consider
\begin{equation*}
\gamma = (\R \times \R_{<0}) \times \gamma_0^- \cup (\R \times \R_{\geq 0}) \times \{ (v,w) \in \R^2 ; w > \max(-v,\sqrt{t}|v|) \},
\end{equation*}
but one cannot choose a Lipschitz continuous cone.
\end{example}

\section{Hyperbolic systems on causal manifolds}\label{section:hyperbolic}

\subsection{The Cauchy problem for $\shd$-modules}

Let $(X,\sho_X)$ be a complex manifold and denote as usual by $\shd_X$ the sheaf of rings of holomorphic (finite order) differential operators on $X$.
References for $\shd$-module theory are made to to~\cite{Ka03}.

Let $\shm$ be a left coherent $\shd_X$-module.
By~\cite{KS90}*{Thm.~11.3.3}, the characteristic variety of $\shm$ is equal to the microsupport of the complex of its holomorphic solutions:
\eqn
&&\chv(\shm)=\SSi(\rhom[\shd_X](\shm,\sho_X)).
\eneqn
Let $Y$ be a complex submanifold of the complex manifold $X$.
One says that $Y$ is non-characteristic for $\shm$ if
\eqn
&&\chv(\shm)\cap T^*_YX\subset T^*_XX.
\eneqn
With this hypothesis, the induced system $\shm_Y$ by $\shm$ on $Y$ is
a coherent $\shd_Y$-module and one has the Cauchy--Kowalesky--Kashiwara theorem~\cite{Ka70}:

\begin{theorem}\label{th:CKK}
Assume $Y$ is non-characteristic for $\shm$.
Then $\shm_Y$ is a coherent $\shd_Y$-module and the morphism
\eq\label{eq:CKK}
&&\rhom[\shd_X](\shm,\sho_X)\vert_Y\to\rhom[\shd_Y](\shm_Y,\sho_Y).
\eneq
is an isomorphism.
\end{theorem}

\subsection{Review on hyperbolic systems}

We shall very briefly recall how to apply the preceding results to hyperbolic systems, using the results of~\cite{KS90}.
For a detailed exposition, we refer to~\cite{Sc13}.

Now let $M$ be a real analytic manifold, say of dimension $n$, and let $X$ be a complexification of $M$.
Consider a vector bundle $\tau\cl E\to M$.
It gives rise to a morphism of vector bundles over $M$, $\tau'\cl TE\to E\times_MTM$ which by duality gives the map $\tau_d\cl E\times_MT^*M\to T^*E$.
By restricting to the zero-section of $E$, we get the map:
\eqn
&& T^*M\into T^*E.
\eneqn
Applying this construction to the bundle $T^*_MX$ above $M$, and using the Hamiltonian isomorphism we get the maps
\eq\label{eq:hyp1}
&&T^*M\into T^*T^*_MX\simeq T_{T^*_MX}T^*X.
\eneq

\begin{definition}
Let $\shm$ be a coherent left $\shd_X$-module.
\banum
\item
We set
\eq\label{eq:hyp2}
&&\hchv_M(\shm)=T^*M\cap C_{T^*_MX}(\chv(\shm))
\eneq
and call $\hchv_M(\shm)$ the \define{hyperbolic characteristic variety} of $\shm$ along $M$.
\item
A vector $\theta\in T^*M$ such that $\theta\notin\hchv_M(\shm)$ is called \define{hyperbolic} with respect to $\shm$.
\item
A submanifold $N$ of $M$ is called \define{hyperbolic} for $\shm$ if
\eq\label{eq:hyp3}
&&T^*_NM\cap\hchv_M(\shm)\subset T^*_MM,
\eneq
that is, any nonzero vector of $T^*_NM$ is hyperbolic for $\shm$.
\item
For a differential operator $P$, we set $\hchv(P) = \hchv_M(\shd_X / \shd_X \cdot P)$.
\eanum
\end{definition}

By construction, $\hchv_M(\shm)$ is a closed conic subset of $T^*M$, invariant by the antipodal map $(x;\xi)\mapsto(x;-\xi)$.

\begin{example}
Assume we have a local coordinate system $(x+\sqrt{-1}y)$ on $X$ with $M=\{y=0\}$ as above and let $(x+\sqrt{-1}y;\xi+\sqrt{-1}\eta)$ be the coordinates on $T^*X$ so that $T^*_MX=\{y=\xi=0\}$.
Let $(x_0;\theta_0)\in T^*M$ with $\theta_0\not=0$.
Let $P$ be a differential operator with principal symbol $\sigma(P)$.
Applying the definition of the normal cone, we find that $(x_0;\theta_0)$ is hyperbolic for $P$ if and only if
\eq\label{eq:hyp4}
&&\left\{\parbox{65ex}{
there exist an open neighborhood $U$ of $x_0$ in $M$ and an open conic neighborhood $\gamma$ of $\theta_0\in\R^n$ such that $\sigma(P)(x;\theta+\sqrt{-1}\eta)\not=0$ for all $\eta\in\R^n$, $x\in U$ and $\theta\in\gamma$.
}\right.
\eneq
As noticed by M.~Kashiwara (see~\cite{BoS73}*{\S~2.}), it follows from the local Bochner's tube theorem that Condition~\eqref{eq:hyp4} can be simplified:
$(x_0;\theta_0)$ is hyperbolic for $P$ if and only if
\eq\label{eq:hyp4K}
&&\left\{\parbox{65ex}{
there exist an open neighborhood $U$ of $x_0$ in $M$ such that $\sigma(P)(x;\theta_0+\sqrt{-1}\eta)\not=0$ for all $\eta\in\R^n$, and $x\in U$.
}\right.
\eneq
Hence, one recovers the classical notion of a (weakly) hyperbolic operator (see~\cite{Le53}).
\end{example}

Now, consider the sheaves
\eq\label{eq:sato1}
&&\sha_M=\sho_X\vert_M,\quad \shb_M=H^n_M(\sho_X)\tens\ori_M\simeq\rsect_M(\sho_X)\tens\ori_M\,[n].
\eneq
Here, $\ori_M$ is the orientation sheaf on $M$ and $n = \dim M$.
The sheaf $\sha_M$ is the sheaf of (complex valued) real analytic functions on $M$ and the sheaf $\shb_M$ is the sheaf of Sato's hyperfunctions on $M$.
Recall that the sheaf $\shb_M$ is flabby and the sheaf $\sha_M$ is $\Gamma$-acyclic, that is, $\rsect(U;\sha_M)$ is concentrated in degree $0$ for each open set $U\subset M$.
Applying~\cite{KS90}*{Cor.~6.4.4}, we get:

\begin{theorem}[see~{KS90}]\label{th:hyp3}
Let $\shm$ be a coherent $\shd_X$-module.
Then
\eqn
&&\SSi(\rhom[\shd_X](\shm,\shb_M))\subset\hchv_M(\shm),\\
&&\SSi(\rhom[\shd_X](\shm,\sha_M))\subset\hchv_M(\shm).
\eneqn
\end{theorem}

In other words, hyperfunction (as well as real analytic) solutions of the system $\shm$ propagate in the hyperbolic directions.

Now we consider the following situation:
$N\into M$ is a real analytic smooth closed submanifold of $M$ of codimension $d$ and $Y\into X$ is a complexification of $N$ in $X$.

The next result was announced without proof in~\cite{Sc13}.
For sake of completeness, we give a proof here.

\begin{lemma}
Let $M,X,N,Y$ be as above and let $\shm$ be a coherent $\shd_X$-module.
Assume that $N$ is hyperbolic for $\shm$.
Then $Y$ is non-characteristic for $\shm$ in a neighborhood of $N$.
\end{lemma}

\begin{proof}
(i)
Consider the exact sequence of vector bundles over $N$:
\eqn
&&0\to T^*_MX\times_MN\times_YT^*_YX\to N\times_YT^*_YX\to T^*_NM\to 0
\eneqn
It defines the inclusion $T^*_NM\subset C_{T^*_MX}(T^*_YX)$ hence the inclusion $T^*_NM\subset C_{T^*_MX}(T^*_YX)\cap T^*M$.
Using local coordinates one checks the equality:
\eq\label{eq:hyp5}
&&C_{T^*_MX}(T^*_YX)\cap T^*M=T^*_NM.
\eneq

\spa
(ii)
Now let $(x;\theta)\in \chv(\shm)\cap T^*_YX\cap(M\times_XT^*X)$ and denote by $L\subset T^*_xX$ the cone $(x;\C^\times\cdot\theta)$.
Using~\eqref{eq:hyp2},~\eqref{eq:hyp3} and~\eqref{eq:hyp5}, we get
\eq\label{eq:hyp6}
&&C_{T^*_MX}(L)\cap T^*M\subset \hchv(\shm)\cap T^*_NM\subset T^*_MM.
\eneq
Choose a local coordinate system $(x,y)$ on $X$ so that $M=\{y=0\}$ and let $(x,y;\xi,\eta)$ denote the associated coordinates.
Identifying $T_{T^*_MX}T^*X$ with $T^*X$, we get
\eqn
&&T^*M=\{(x,0;\xi,0)\},\quad T^*_MX=\{(x,0;0,\eta)\},\\
&&C_{T^*_MX}(L)=T^*_MX+L.
\eneqn
Then~\eqref{eq:hyp6} implies $\theta=0$.
\end{proof}

The following result is easily deduced from Theorem~\ref{th:hyp3}.
See~\cite{Sc13} for details.

\begin{theorem}\label{th:hyp4}
Let $M$ be a real analytic manifold, $X$ a complexification of $M$, $\shm$ a coherent $\shd_X$-module.
Let $N\into M$ be a real analytic smooth closed submanifold of $M$ and $Y\into X$ is a complexification of $N$ in $X$.
We assume
\eq\label{hyp:nonhyp}
&& T^*_NM\cap\hchv_M(\shm)\subset T^*_MM,
\eneq
that is, $N$ is hyperbolic for $\shm$.
Then $Y$ is non-characteristic for $\shm$ in a neighborhood of $N$ and
the isomorphism~\eqref{eq:CKK} induces the isomorphism
\eq\label{eq:CKhyp}
&&\rhom[\shd_X](\shm,\shb_M)\vert_N\isoto \rhom[\shd_Y](\shm_Y,\shb_N).
\eneq
\end{theorem}

In other words, the Cauchy problem in a neighborhood of $N$ for hyperfunctions on $M$ is well-posed for hyperbolic systems.

Note that for $f\in \ext[\shd_X]{j}(\shm,\shb_M)$, its wavefront set does not intersect $T^*_YX\cap T^*_MX$ and thus its restriction to $\ext[\shd_Y]{j}(\shm_Y,\shb_N)$ is well-defined.
In particular, if $\shm=\shd_X/\shd_X\cdot P$, this means that if $f$ is a hyperfunction on $M$ defined in a neighborhood of $N$ solution of $Pf=0$, then its wavefront set does not intersect $T^*_YX\cap T^*_MX$ and therefore its restriction is well-defined (see~\cite{SKK73}).

\subsection{Hyperbolic systems on causal manifolds}

\begin{theorem}\label{th:HypCPgl}
Let $M$ and $S$ be a real analytic manifolds, $X$ a complexification of $M$, $\shm$ a coherent $\shd_X$-module.
Let $I$ be a finite set and let $\{\gamma_{M,i}\}_{i\in I}$ and $\{\gamma_{S,i}\}_{i\in I}$ be two families of open convex cones in $TM$ and $TS$ respectively.
Let $f\cl M\to S$ be a morphism of manifolds which defines morphisms of causal manifolds $f\cl(M,\gamma_{M,i})\to(S,\gamma_{S,i})$.
For each $i\in I$ let $\preceq_i$ be a closed preorder on $M$, causal for $\gamma_i$.
Assume
\banum
\item
$f$ is submersive and surjective,
\item
$S$ is contractible,
\item
for any compact $K\subset M$ and all $i\in I$, the map $f$ is proper on $\pasii{K}$,
\item
$\hchv(\shm)\cap\lambda_{M,i}\subset T^*_MM$ for all $i\in I$,
\item
$\bigcup_i\Int(\lambda_{S,i})=T^*S\setminus T^*_SS$.
\eanum
Let $a\in S$ and assume that $N\eqdot\opb{f}(\{a\})$ is real analytic.
Let $Y$ be a complexification of $N$ in $X$.
Then the isomorphism~\eqref{eq:CKK} induces the isomorphism
\eq\label{eq:CKhypGl}
&&\RHom[\shd_X](\shm,\shb_M)\isoto \RHom[\shd_Y](\shm_Y,\shb_{N}).
\eneq
\end{theorem}

Recall that on a topological space $M$, $\RHom\simeq \rsect(M;\scbul)\circ \rhom$.

\begin{proof}
Set $F = \rhom[\shd_X](\shm,\shb_M)$.
Then $\SSi(F)\subset \hchv(\shm)$ by Theorem~\ref{th:hyp3}.
Applying Corollary~\ref{cor:GP3}, we get $\rsect(M;F)\simeq\rsect(N;F\vert_{N})$ and, applying Theorem~\ref{th:hyp4}, we get $\rsect(N;F\vert_{N})\simeq \RHom[\shd_Y](\shm_Y,\shb_{N})$.
\end{proof}

\begin{corollary}\label{cor:CKGcausal}
Let $(M,\gamma,\tim)$ be a G-causal manifold and assume that $M$ is real analytic.
Let $\shm$ be a coherent $\shd_X$-module satisfying $\hchv(\shm)\cap\lambda\subset T^*_MM$.
\banum
\item
Let $A$ be a closed subset satisfying either $A=\futccc{A}$ and $A\subset \opb{\tim}([a,+\infty))$ or $A=\pasccc{A}$ and $A\subset \opb{\tim}((-\infty,a])$ for some $a\in\R$.
Then $\RHom[\shd_X](\shm,\sect_A\shb_M)\simeq0$.
In particular, hyperfunction solutions of the system $\shm$ defined on $M\setminus A$ extend uniquely to the whole of $M$ as hyperfunction solutions of the system.
\item
Let $N=\opb{\tim}(0)$ and assume that $N$ is real analytic.
Let $Y$ be a complexification of $N$ in $X$.
Then the restriction morphism
$\RHom[\shd_X](\shm,\shb_M)\to\RHom[\shd_Y](\shm_Y,\shb_N)$ is an isomorphism.
In other words, the Cauchy problem for hyperfunctions with initial data on $N$ is globally well-posed.
\eanum
\end{corollary}

Note that if $K$ is compact, then $A= \futccc{K}$ or $A=\pasccc{K}$ will satisfy the condition in (a).

\begin{proof}
(a)
Apply Theorem~\ref{th:GglobalCP1}~(i).

\spa
(b)
Apply Theorem~\ref{th:GglobalCP1}~(ii) or Theorem~\ref{th:HypCPgl}.
\end{proof}

\begin{remark}
Theorem~\ref{th:HypCPgl} and Corollary~\ref{cor:CKGcausal} remain true when replacing the sheaves of hyperfunctions by those of real analytic functions, that is, replacing $\shb_M$ and $\shb_N$ with $\sha_M$ and $\sha_N$, respectively.
\end{remark}

\subsection{Examples}\label{subsection:example2}

In this subsection, all manifolds will be real (or complex) analytic and the differential operators we consider will have analytic coefficients.

We shall often assume that $M=N\times\R$.
In this situation, for $(x,t)\in M$, the vector $(x,t;dt)=(x,t;0,1)$ is well-defined in $T^*M$.

\subsubsection{Product with a compact Cauchy hypersurface}

Let us translate Corollary~\ref{cor:CKGcausal} in a particular situation.

\begin{proposition}\label{pro:wave2}
Let $N$ be a real analytic compact manifold and let $M = N \times \R$.
Let $P$ be a differential operator on $M$ of order $m$ which is hyperbolic in the codirection $(x,t; dt)$ for all $(x,t) \in M$.
\banum
\item
The Cauchy problem
\eq\label{eq:CauchyProb}
&&\left\{\parbox{60ex}{
$Pf=0$\\
$(f, \dots, \partial_t^{m-1} f)|_{t=0} = (h_0, \dots, h_{m-1})$
}\right.
\eneq
is globally well-posed for hyperfunctions and for analytic functions.
In other words, for any $h=(h_0, \dots, h_{m-1})$ in $\shb(N)^m$ \lp resp.\ $\sha(N)^m$\rp\, there exists a unique $f\in\shb(M)$ \lp resp.\ $f\in\sha(N)$\rp\, solution of~\eqref{eq:CauchyProb}.
\item
The operator $P$ is a surjective endomorphism of $\sect(M; \shb_M)$ and of $\sect(M; \sha_M)$.
\item
Any hyperfunction \lp resp.\ analytic function\rp\, $f$ solution of the equation $Pf = 0$ defined on $N \times (a,b)$ with $-\infty\leq a<b\leq+\infty$ extends uniquely as a hyperfunction \lp resp.\ analytic function\rp\, solution of this equation on $M$.
\eanum
\end{proposition}

\begin{proof}
(a)-(b)
Consider the open convex cone $\gamma = TN \times \{ (t;v) \in T\R ; v>0 \}$ in $TM$.
The projection on the second factor, $q \colon M \to \R, (x,t) \mapsto t$, is a Cauchy time function on $(M,\gamma)$.
One has $\gamma^\circ = T^*_NN \times \{ (t;\tau) \in T^*\R ; \tau \geq 0 \}$.
Since $(x,t; 0, 1)$ is a hyperbolic codirection for all $(x,t) \in M$, one has $\hchv(P) \cap \gamma^\circ \subset T^*_MM$, so we can apply Corollary~\ref{cor:CKGcausal}~(b).

Since $\RHom[\shd_Y](\shm_Y,\shb_N)\simeq \rsect(N;\shb_N^{m})$ is concentrated in degree $0$, we get that $\Ext[\shd]{1}(\shm,\shb_M)=0$, that is, the operator $P\cl \shb(M)\to\shb(M)$ is surjective and its kernel is isomorphic to $\shb(N)^m$, the isomorphism being given (for example) by $f\mapsto (f, \dots, \partial_t^{m-1} f)|_{t=0}$.
Indeed, an isomorphism $\shm_Y\simeq \shd_X/t\cdot\shd_X+\shd_X\cdot P$ is described as follows: any $R(x,t;\partial_x,\partial_t)\in\shd_X\vert_Y$ may be written uniquely as
\eqn
&&R(x,t;\partial_x,\partial_t)=t\cdot Q + S\cdot P+\sum_{j=0}^{m-1}R_j(x,\partial_x)\cdot \partial_t^j
\eneqn
by the Sp\"ath--Weierstrass division theorem (see~\cites{Ka70,SKK73} for details).

\spa
(c) follows from Corollary~\ref{cor:CKGcausal}~(a).
\end{proof}

\begin{example}\label{exa:loop1}
Let $P$ be a differentoial operator of order $2$ such that
\eq\label{eq:hypforR}
&&\left\{\parbox{70ex}{
\vspace{1.ex}$P = \partial_t^2 - R$,\\
\vspace{1.ex}$ \sigma_2(R)\vert_{T^*_MX}\leq0$,\\
$\sigma_2(R)$ does not depend on $\tau$.
}\right.\eneq
Then $P$ is hyperbolic in the codirections $(x,t;\pm dt)$ for all $(x,t) \in M$.
Indeed, choose a local coordinate system $(z;\zeta)$ on $T^*Y$, $z=x+\sqrt{-1}y$, $\zeta=\xi+\sqrt{-1}\eta$.
Denote by $(t+\sqrt{-1}t';\tau+\sqrt{-1}\tau')$ the coordinates on $T^*\C$.
Hence, $(x,t;\sqrt{-1}\eta,\sqrt{-1}\tau')$ is a local coordinate system on $T^*_MX$.
Then, denoting by $\sigma(P)$ the principal symbol of $P$,
\eqn
\sigma(P)(x,t;\sqrt{-1}\eta,\sqrt{-1}\tau' +\theta)&=&(\theta+\sqrt{-1}\tau')^2-\sigma_2(R)(x,t;\sqrt{-1}\eta)\\
&=&\theta^2-\tau^{\prime 2}+\sigma_2(R)(x,t;\eta)+ 2\sq\theta\tau'
\neq0\mbox{ for }\theta\neq0.
\eneqn
If $(g_t)_{t \in \R}$ is an analytic family of Riemannian metrics on $N$ and $(\Delta_t)_{t \in \R}$ are the associated Laplace--Beltrami operators, then
\begin{equation}
P = \partial_t^2 - \Delta_t
\end{equation}
is such an example.

Being hyperbolic in a given codirection depends only on the top-order part of the operator.
If a differential operator is hyperbolic in a given codirection, then so are its powers.
Therefore, all operators of the form $P^r+Q$ where $Q$ is a differential operator on $M$ of order at most $2r-1$ are also examples.
\end{example}

\begin{remark}
Similar results do not hold in general with the sheaves of distributions or of $C^\infty$-functions.
For example, it is well-known since Hadamard that the Cauchy problem is not well-posed in the space of $C^\infty$-functions on $\R^2$ for the operator $\partial_t^2-\partial_x$.
However, if one assumes that the operator $R$ in Example~\ref{exa:loop2} is elliptic, then the operator $P = \partial_t^2 - R$ is hyperbolic in the classical sense and the Cauchy problem is well-posed in the spaces of $C^\infty$-functions and of distributions on $N \times \R$.
\end{remark}

\begin{example}\label{exa:loop2}
Let us particularize Example~\ref{exa:loop1} to the case $N=\BBS^1$, hence $M=\BBS^1\times\R$.
We define $\gamma,\preceq,\tim$ as above and we denote denote by $x$ a coordinate on $\BBS^1$ (hence, $x+2\pi=x$).
Then the path $[0,2\pi]\ni s\mapsto (s,0)\in M$ is a causal loop.

Consider the differential operator $P(x,t;\partial_x,\partial_t)=\partial_t^2-\partial_x^2$.
The Cauchy problem~\eqref{eq:CauchyProb} (with $m=2$) is globally well-posed in various spaces of functions or generalized functions.
In fact, writing $f(x,t)=f_0(x+t)+f_1(x-t)$, we get
\eqn
&& f_0(x) + f_1(x)=h_0(x),\quad f'_0(x)-f'_1(x)=h_1(x).
\eneqn
Therefore, $2f_0=h_0+\int h_1$ and $2f_1(x)=h_0-\int h_1$.
Note that $\int h_1$ is not necessarily periodic, but replacing $f(x,t)$ with $f(x,t) + (\int_0^{2\pi} h_1) t$, we may assume from the beginning that $\int_0^{2\pi} h_1=0$ so that $\int h_1$ is $2\pi$-periodic.
If one works in the space of real analytic functions or in the space of hyperfunctions, this result is in accordance with Corollary~\ref{cor:CKGcausal}.
\end{example}

\begin{example}\label{exa:loop3}
We consider Example~\ref{exa:loop2} and replace $\R$ with coordinate $t$ with the circle $\R/a\cdot\Z$ for some $a>0$.
Hence, now $M$ is a torus.
Set $S= \R/a\cdot\Z$ and keep the notation of the previous example.
Hence $\tim\cl M\to S$ is a submersive morphism of causal manifolds, but not a time function since $S\neq\R$.
Moreover, the Cauchy problem~\eqref{eq:CauchyProb} (with $m=2$) is not globally well-posed, except for $a=2\pi$.
Note that Theorem~\ref{th:HypCPgl} does not apply since $S$ is not contractible.
\end{example}

\subsubsection{Complex time}

\begin{example}\label{exa:loop1C}
In this example, we treat the case where the time is complex.
For simplicity, we restrict ourselves to an elementary situation.

Let $N$ be a real analytic compact manifold and let $M=N\times\C$.
Let $Y$ be a complexification of $N$.
We denote by $w=t+\sq t'$ the complex coordinate on $\C$ and by $(w;\tau+\sq\tau')$ the coordinates on $T^*\C$.
Let $M=N\times\C$ viewed as a real manifold.
Consider the left ideal $\shi$ and the left $\shd_M$-module $\shm$:
\eqn
&&\shi= \shd_M\cdot P+\shd_M\cdot\ol\partial_w,\quad\shm=\shd_M/\shi,
\eneqn
where $\ol\partial_w$ is the Cauchy--Riemann operator on $\C$ and
\eqn
&&P = \partial_w^2 - R
\eneqn
where $R$ is a differential operator on $M$ (holomorphic in $w$) of order $\leq 2$ whose symbol of order 2 depends neither on $w$ nor on $\tau+\sq\tau'$ and satisfies
\eqn
&&\sigma_2(R)\vert_{T^*_NY} \leq 0.
\eneqn
Then any $(x,w;0,\theta + \sq \theta' )\in T^*M$, $ \theta \neq 0$ is hyperbolic for $\shm$.
Indeed, the system of equations
\eqn
&&\sigma(\ol{\partial}_w) \big( x,w; \sq\eta, (\theta + \sq \theta') + \sq(\tau + \sq \tau') \big) = 0,\\
&&\sigma(P) \big( x,w; \sq\eta, (\theta + \sq \theta') + \sq(\tau + \sq \tau') \big) = 0
\eneqn
is equivalent to
\eqn
(\theta + \sq \tau)^2 - \sigma_2(R)(x; \sq \eta) = 0
\eneqn
which has no solutions for $\theta \neq 0$.
In other words,
\eq\label{eq:hchvExa3}
&&\hchv_M(\shm)\cap (T^*_NN\times T^*\C)\subset T^*_NN\times (\C\times\sq \R).
\eneq
Since the codirections $\sq\theta'$ are not hyperbolic for the system, we cannot apply Theorem~\ref{th:HypCPgl} to solve the Cauchy problem with data on the submanifold $N\times\{0\}$ of $M$ and, indeed, one easily sees that the Cauchy problem
\eqn
&&\left\{\parbox{60ex}{
$\;Pf=0$,\vspace{1ex}\\
$\ol{\partial_w}f=0$,\vspace{1ex}\\
$(f, \partial_w f)|_{w=0} = (h_0,h_{1})$
}\right.
\eneqn
is not well posed.
However, one has propagation results:
\eq\label{eq:Ctimespropag}
&&\left\{
\parbox{75ex}{
Let $\Omega_0$ be an open subset of $\C$ whose intersection with any line $\R+\sq a$, $a \in \R$, is connected and let $\Omega_1=\Omega_0+\R_{\geq0}\times\{0\}$.
Set $U_i=N\times \Omega_i$ ($i=0,1$).
Then one has the restriction isomorphism \\
$\rsect(U_1;\rhom[\shd_X](\shm,\shb_M))\isoto \rsect(U_0;\rhom[\shd_X](\shm,\shb_M))$.
}
\right.
\eneq
To prove~\eqref{eq:Ctimespropag}, one may proceed as follows.
Let $\tim\cl N\times\C\to\C$ denote the projection and set $F=\roim{\tim}\rhom[\shd_X](\shm,\shb_M)$.
Since $q$ is proper, it follows from~\eqref{eq:hchvExa3} and Theorem~\ref{th:hyp3} that $\SSi(F)\subset \C\times\sq\R\subset T^*\C$.
Then the isomorphism $\rsect(\Omega_1;F)\isoto\rsect(\Omega_0;F)$ follows from~\cite{KS90}*{Prop.~5.2.1}.
Indeed, with the notations of loc.\ cit., choose $\gamma=\R_{\leq0}\times \{0\}$ and $U=\Omega_1$.
Then, for any $x\in\Omega_1$, the set $(x+\gamma) \setminus \Omega_0$ is compact.

Of course, one may interchange the cones $\R_{\geq0}\times\{0\}$ and $\R_{\leq0}\times\{0\}$.
For an open subset $\Omega$ of $\C$ whose intersection with any line $\R+\sq a$, $a\in\R$ is connected, denote by $\tw\Omega$ the open set $\Omega+\R\times\{0\}$ and for $U=N\times\Omega$ set $\tw U=N\times\tw\Omega$.
One gets the isomorphism
\eqn
&&\rsect(\tw U;\rhom[\shd_X](\shm,\shb_M))\isoto \rsect(U;\rhom[\shd_X](\shm,\shb_M)).
\eneqn
\end{example}

\subsubsection{Product with a Riemannian hypersurface}

We shall need the following result.

\begin{lemma}[\cite{BEE96}*{Thm~3.66}]\label{lem:compl}
Let $(N,g)$ be a Riemannian manifold and let $f \colon \R \to \R_{>0}$ be a smooth function.
Then the Lorentzian spacetime $(N \times \R, dt^2 - f(t)g)$ is globally hyperbolic if and only if $g$ is complete, and in this case, the projection on the second factor $q \colon N \times \R \to \R$ is a Cauchy time function.
\end{lemma}

\begin{example}\label{exa:loop4}
Let $N$ be a real analytic manifold and set $M = N\times\R$ as in Example~\ref{exa:loop1} but now, {\em we do not assume any more that $N$ is compact}.

We still denote by $t$ a coordinate on $\R$, by $(t;w)$ the coordinates on $T\R$ and by $(t;\tau)$ the coordinates on $T^*\R$
and we still consider a differential operator $P = \partial_t^2 - R$ as in~\eqref{eq:hypforR}.
Since $N$ is no more assumed to be compact, we need another hypothesis:
\eq\label{eq:hypR8}
&&\parbox{65ex}{there exist a smooth function $f\cl\R\to\R_{>0}$ and a smooth complete Riemannian metric $g$ on $N$ such that $\sigma_2(R)(x,t;\xi)\leq f(t)\vert\xi\vert_{g_x}^2$.
}\eneq
Note that this condition is automatically satisfied if $N$ is compact.
We want to solve the homogeneous Cauchy problem~\eqref{eq:CauchyProb} for hyperfunctions (or analytic functions).
If we chose $\gamma$ as in the proof of Proposition~\ref{pro:wave2}, then the projection on the second factor would not be proper on the sets $\futo{x}$ for any $x \in M$.
We set
\begin{equation}
\gamma = \{(x,t;v,w)\in TM; w > 1/(2 f(t)) |v|_g\}.
\end{equation}

One has $\gamma^\circ = \{(x,t;\xi,\tau)\in T^*M; \tau \geq 2 f(t) |\xi|_g\}$.
By  Lemma~\ref{lem:quadra0}, $\hchv(P) \cap \gamma^\circ \subset T^*_MM$.

One checks that $\gamma$ is the future cone of the Lorentzian spacetime $(M,dt^2-(1/2f(t)) g)$, which is globally hyperbolic by Lemma~\ref{lem:compl}.
Therefore, $(M,\gamma,q)$ is a G-causal manifold and we can apply Corollary~\ref{cor:CKGcausal} which asserts that the Cauchy problem~\eqref{eq:CauchyProb} for hyperfunctions (or analytic functions) is globally well-posed.

As in the compact case, if $(g_t)_{t \in \R}$ is an analytic family of \emph{complete} Riemannian metrics on $N$ and $(\Delta_t)_{t \in \R}$ are the associated Laplace--Beltrami operators, then the operator $P = \partial_t^2 - \Delta_t$ is such an example.
\end{example}

\subsubsection{The wave operator on a globally hyperbolic Lorentzian spacetime}

\begin{definition}
Let $(M,g)$ be a real analytic Lorentzian manifold.
The \define{wave operator} is defined by
\begin{equation}
\square = - \operatorname{div} \grad.
\end{equation}
A \define{wave-type operator} on $(M,g)$ is a differential operator $P$ whose symbol satisfies
\begin{equation}
\sigma(P)(x; \xi) = |\xi|_{g_x}^2
\end{equation}
for $(x;\xi) \in T^*M$.
\end{definition}

A standard calculation shows that the wave operator on a Lorentzian manifold is a wave-type operator on that Lorentzian manifold.

\begin{lemma}\label{lem:hypchar}
Let $(M,g)$ be a Lorentzian spacetime and let $P$ be a wave-type operator, then $\hchv(P) \cap \Int(\gamma_g^\circ) \subset T^*_MM$.
\end{lemma}

\begin{proof}
This follows directly from Lemma~\ref{lem:quadra0}.
\end{proof}

\begin{theorem}\label{th:wave1}
Let $(M,g)$ be a real analytic globally hyperbolic Lorentzian spacetime and let $P$ be a wave-type operator on $M$.
Let $N \subset M$ be a real analytic Cauchy hypersurface and let $v$ be an analytic vector field defined in a neighborhood of $N$ and transversal to $N$.
Then the Cauchy problem
\eq\label{eq:CauchyProb2}
&&\left\{\parbox{60ex}{
$Pf = 0$\\
$(f|_N, v (f)\vert_N)= (h_0, h_1)$
}\right.
\eneq
with $h_0, h_1 \in \shb(N)$ \lp resp.\ $\sha(N)$\rp\ has a unique global solution in $\shb(M)$ \lp resp.\ $\sha(M)$\rp.

Furthermore, the operator $P$ is a surjective endomorphism of $\shb(M)$ and of $\sha(M)$.
\end{theorem}

\begin{proof}
Since global hyperbolicity is a stable property, there exists a Lorentzian metric $\tilde{g}$ on $M$ such that $(M,\tilde{g})$ is globally hyperbolic and $\ol{\gamma} \subset \tilde{\gamma} \cup \{0\}$, so $\tilde{\gamma}^\circ \subset \Int(\gamma^\circ) \cup \{0\}$.
By Lemma~\ref{lem:hypchar}, $\hchv(P) \cap \tilde{\gamma}^\circ \subset \hchv(P) \cap (\Int(\gamma^\circ) \cup T^*_MM) \subset T^*_MM$.

Let $q$ be a Cauchy time function such that $q^{-1}(0) = N$.
We apply Corollary~\ref{cor:CKGcausal} to the G-causal manifold $(M,\tilde{\gamma},q)$ and the $\shd_X$-module $\shd_X / \shd_X P$.
Then the proof goes as for Proposition~\ref{pro:wave2}.
\end{proof}

One shall notice that in Theorem~\ref{th:wave1}, there is no assumption that the initial data be compactly supported.

\appendix
\section{Appendix: normal cones}
References are made to~\cite{KS90}.

Let $A,B$ be two subsets of $M$.
The \define{Whitney cone} $C(A,B)$ (see~\cite{KS90}*{Def.~4.1.1}) is a closed conic subset of $TM$.
In a chart at $x_0\in M$, it is described as follows.
\eq\label{eq:normalC}
&&v\in C_{x_0}(A,B)\subset T_{x_0}M\Leftrightarrow\left\{\parbox{45ex}{
there exists a sequence $\{(x_n,y_n,\lambda_n)\}_n\subset A\times B\times\R_{>0}$ such that\\
$x_n\to[n]x_0$, $y_n\to[n]x_0$, $\lambda_n(x_n-y_n)\to[n]v$.
}\right.
\eneq
For short, we define $C_x(A)=C_x(A,\{x\})\subset T_xM$ and often identify it with $C(A, \{x\}) \subset TM$.
This is a closed cone of $T_xM$, the set of limits when $y\in A$ goes to $x$ of half-lines issued at $x$ and passing through $y$.
More generally, for $N$ a smooth submanifold of $M$, the set $C(A,N)$ satisfies $N \times_M C(A,N) +TN = N \times_M C(A,N)$ and one denotes by $C_N(A)$ the image of $N \times_M C(A,N)$ in $T_NM = (N \times_M TM)/TN$.

Let us recall without proof some elementary properties of Whitney cones that we will need later.

\begin{proposition}\label{prop:cones0}
Let $L,M,N$ be manifolds and let $g\cl L\to M$ and $f\cl M\to N$ be morphisms of manifolds.
Let $A, A_1, A_2, B \subset M$.
Then
\bnum
\item the Whitney cone $C(A, B) \subset TM$ is a closed cone,
\item $C(A, B) = - C(B, A)$,
\item if $A_1 \subset A_2$, then $C(A_1, B) \subset C(A_2, B)$,
\item $C(A_1 \cup A_2, B) = C(A_1, B) \cup C(A_2, B)$,
\item $C(\clos{A}, B) = C(A, B)$,
\item $C(A,B)\cap T_MM = \clos{A} \cap \clos{B}$,
\item $C(g^{-1}(A),g^{-1}(B)) \subset \opb{(Tg)}C(A,B)$,
\item $Tf(C(A,B))\subset C(f(A),f(B))$.
\enum
\end{proposition}
Let $A$ be a subset of $M$.
Recall~(\cite{KS90}*{Def.~5.3.6}) that the \define{strict normal cone} of $A$ is the set
\eq\label{eq:strictNC}
&&N(A) = TM \setminus C(M \setminus A, A).
\eneq
This is an open convex cone of $TM$ (see below).
In a chart at $x \in M$, one has:
\eq\label{eq:strictNC2}
&&(x,v)\in N(A)\Leftrightarrow\left\{\parbox{40ex}{there exists an open cone $\gamma_0$ with $v\in\gamma_0$ and an open neighborhood $U$ of $x$ such that $U\cap(U\cap A+\gamma_0)\subset A$.
}\right.
\eneq
As usual, for $x \in M$, one sets $N_x(A)= T_xM \cap N(A)$.

\begin{proposition}\label{pro:NinC}
Let $A\subset M$ and let $x \in \ol{A}$.
Then $\ol{N_x(A)}\subset C_x(A)$.
\end{proposition}

\begin{proof}
Let us choose a chart at $x$.
Since $C_x(A)$ is closed, it is enough to check the inclusion $N_x(A) \subset C_x(A)$.
Let $v \in N_x(A)$.
There is a neighborhood $U$ of $x$ and a conic neighborhood $\gamma_0$ of $v$ such that $U \cap (U \cap A + \gamma_0) \subset A$.
Since $x \in \ol{A}$, there is a sequence $x_n \to x$ with $x_n \in A$.
Let $c_n>0$ be a sequence with $c_n\to[n]+\infty$ and $c_n(x-x_n)\to[n]0$.
Set $y_n=x_n+ \opb{c_n}v$.
Then for $n$ large enough, $y_n \in A$ and
\eqn
&&c_n(y_n-x) = c_n (y_n - x_n) + c_n (x_n - x) = v + c_n (x_n - x) \to[n] v.
\eneqn
\end{proof}

For short, we set
\eq\label{eq:outgNC}
&&D(A) = C(M \setminus A, A)
\eneq
and we call $D(A)$ the \define{cone of outgoing vectors of $A$}.
Most of the following properties are direct consequences of the corresponding properties of the Whitney cone.
We gather them in a proposition for later reference.

\begin{proposition}\label{prop:cones}
Let $L,M,N$ be manifolds and let $g\cl L\to M$ be a morphism of manifolds.
Let $A, A_1, A_2 \subset M$ and $B\subset N$.
Then
\bnum
\item
$N(A)$ is an open convex cone,
\item
$N(M \setminus A ) = N(A)^a$,
\item
$N(\varnothing) = N(M) = TM$,
\item
$N_x(A)=T_xM$ if and only if $x\notin\partial A$,
\item
$N(A_1) \cap N(A_2)\subset N(A_1 \cup A_2)$
and $N(A_1) \cap N(A_2)\subset N(A_1 \cap A_2)$,
\item
$N(A) \subset N(\clos{A})$ and $N(A) \subset N(\Int{A})$,
\item
if $N_x(A) \neq \varnothing$ for all $x \in \partial A$, then $\clos{\Int{A}} = \clos{A}$ and $\Int{\clos{A}} = \Int{A}$,
\item
if $\clos{A_1} = \clos{A_2}$ and $\Int{A_1} = \Int{A_2}$, then $N(A_1) = N(A_2)$,
\item
$\opb{(Tg)}(N(A))\subset N(\opb{g}(A))$,
\item
$N(A)\times N(B)\subset N(A\times B)$.
\enum
\end{proposition}

\begin{proof}
(i)
The set $N(A)$ is an open cone by its definition.
Let us choose a chart at $x \in M$.
Let $v, w \in T_xM$ with $v+w \in D_x(A)$.
Then there is a sequence $\{(x_n,y_n,c_n)\}_n$ in $A\times(M\setminus A)\times\R_{>0}$ with $x_n\to[n]x$, $y_n\to[n]x$ and $c_n (y_n - x_n) \to[n] v+w$.
There are infinitely many $x_n + (1/c_n) v$ contained either in $A$ or in $M\setminus A$.
In the second case, $v \in D_x(A)$, and in the first case, $w \in D_x(A)$.
This shows that if both $v$ and $w$ belong to $N_x(A)$, then so does $v+w$.

\spa
(ii)--(iv) are obvious.

\spa
(v)
We have
\eqn
D(A_1 \cup A_2)&= &C(M \setminus (A_1 \cup A_2), A_1 \cup A_2)\\
&=& C((M \setminus A_1) \cap (M \setminus A_2), A_1) \cup C((M \setminus A_1) \cap (M \setminus A_2), A_2)\\
&\subset& C(M \setminus A_1, A_1) \cup C(M \setminus A_2, A_2)
= D(A_1) \cup D(A_2).
\eneqn
The second inclusion follows, using (ii).

\spa
(vi)
We have
\eqn
D(\clos{A}) &=& C(\clos{A}, M \setminus \clos{A}) \\
&=& C(A, M \setminus \clos{A}) \subset C(A, M \setminus A) = D(A).
\eneqn
This proves the first inclusion.
One deduces the second inclusion by using (ii).

\spa
(vii)
Let $x\in\partial A$ and assume that $N_x(A)\neq \varnothing$.
We shall prove that $x\in\ol{\Int A}$.
We choose a chart at $x$.
Let $v \in N_x(A)$.
Then there exist a neighborhood $V$ of $x$ and a conic open neighborhood $C$ of $v$ such that $V \cap (x + C) \subset A$.
Hence, there exists a sequence $\{ t_n\}_n$, $t_n>0$, $t_n \to[n] 0$ such that $x + t_n v \in \Int{A}$.
Therefore, $x \in \clos{\Int{A}}$.
The other inclusion follows, using (ii).

\spa
(viii)
By Proposition~\ref{prop:cones}~(v), using $\clos{M \setminus A} = M \setminus \Int{A}$, we get
\eqn
&&D(A) = C(M \setminus A, A) = C(\clos{M \setminus A}, \clos{A}) = C(M \setminus \Int{A}, \clos{A}).
\eneqn
Hence, $D(A)$ depends only on $\ol A$ and $\Int A$.

\spa
(ix)
By Proposition~\ref{prop:cones}~(vii), we have
\eqn
D(f^{-1}(A)) &=& C(M \setminus f^{-1}(A), f^{-1}(A))\\
&=& C(f^{-1}(L \setminus A), f^{-1}(A))\\
&\subset &\opb{Tf}(C(L \setminus A, A)) = \opb{Tf}(D(A)),
\eneqn
and the result follows.

\spa
(x)
Let $(v_1, v_2)\in T_{(x,y)}(M \times N)$.
We choose local charts centered at $x$ and $y$.
Then $(x, y; v_1, v_2)\in N(A) \times N(B)$ if and only if there exist neighborhoods $U$ of $x$, $V$ of $y$ and conic open neighborhoods $\gamma_1$ of $v_1$ and $\gamma_2$ of $v_2$, such that, setting $W=U\times V$
\eqn
&&W \cap \bl(W \cap A \times B) + \gamma_1 \times \gamma_2 \br \subset A \times B.
\eneqn
Therefore, $(v_1, v_2) \in N_{(x,y)}(A \times B)$.
\end{proof}

\begin{remark}
One could improve Proposition~\ref{prop:cones}~(x) and show that for $(x,y)\in\ol{A\times B}$, $N_{(x,y)}(A\times B)=N_x(A)\times N_y(B)$ but we do not use this result.
\end{remark}

\providecommand{\bysame}{\leavevmode\hbox to3em{\hrulefill}\thinspace}

\vspace*{1cm}
\noindent
\parbox[t]{21em}
{\scriptsize{

Beno{\^i}t Jubin\\
Sorbonne Universit{\'e}s, UPMC Univ Paris 6\\
Institut de Math{\'e}matiques de Jussieu\\
e-mail: benoit.jubin@imj-prg.fr\\

\vspace{1ex}
\noindent
Pierre Schapira\\
Sorbonne Universit{\'e}s, UPMC Univ Paris 6\\
Institut de Math{\'e}matiques de Jussieu\\
e-mail: pierre.schapira@imj-prg.fr\\
http://www.math.jussieu.fr/\textasciitilde schapira/
}}
\end{document}